%% file: 3DSeibergWittenCompactness.tex
\documentclass[leqno,11pt]{article}

\input{preamble/preamble}
\newcommand{\Tstar}{T^*\!}

\author{
  Thomas Walpuski 
  \and
  Boyu Zhang 
}
\title{
  On the compactness problem for a family of generalized Seiberg--Witten equations in dimension three
}
\date{2020-09-01}

\begin{document}

\maketitle

\begin{abstract}
  We prove an abstract compactness theorem for a family of generalized Seiberg--Witten equations in dimension three.
  This result recovers Taubes' compactness theorem for stable flat $\PSL_2(\C)$--connections \cite{Taubes2012} as well as the compactness theorem for Seiberg--Witten equations with multiple spinors \cite{Haydys2014}.
  Furthermore, this result implies a compactness theorem for the ADHM$_{1,2}$ Seiberg--Witten equation, which partially verifies a conjecture by \citet[Conjecture 5.26]{Doan2017d}.
\end{abstract}

\input{Introduction}
\input{LichnerowiczWeitzenbock}
\input{RegularityScale}
\input{ProofOfAbstractCompactnessTheorem}
\input{ProofOfADHM12SeibergWittenCompactness}

\printreferences

\end{document}


%% file: preamble/preamble.tex
\usepackage[utf8]{inputenc}
\usepackage[T1]{fontenc}
\usepackage{microtype}

\usepackage[a4paper]{geometry}

\ifdefined\screenview
  \edef\mtht{\the\textheight}
  \edef\mtwd{\the\textwidth}
\fi

\usepackage[dvipsnames]{xcolor}

\usepackage{amsmath}
\usepackage{amsthm}
\usepackage{tikz-cd}
\usetikzlibrary{arrows} 
\tikzset{
  commutative diagrams/.cd, 
  arrow style=tikz, 
  diagrams={>=stealth}
}
\usetikzlibrary{matrix,decorations.pathreplacing,calc}

\usepackage{textcomp}
\usepackage[sb]{libertine}
\usepackage[varqu,varl]{zi4}%
\usepackage[libertine,bigdelims,vvarbb]{newtxmath}
\usepackage[supstfm=libertinesups,supscaled=1.2,raised=-.13em]{superiors}
\useosf

\usepackage[scr=boondox,cal=euler]{mathalfa}

\usepackage{marvosym}

\usepackage{slashed}
\usepackage{esint} 
\usepackage[english]{babel}

\usepackage{imakeidx}
\indexsetup{noclearpage}
\makeindex[intoc]

\usepackage{csquotes}
\usepackage[
  backend=biber,
  hyperref=true,
  backref=true,
  isbn=false,
  doi=true,
  natbib=true,
  eprint=true,
  useprefix=true,
  maxcitenames=99,
  maxbibnames=99,  
  maxalphanames=99, 
  minalphanames=99,
  safeinputenc,
  style=alphabetic,
  citestyle=alphabetic,
  block=space,
  datamodel=preamble/ext-eprint
]{biblatex}
\usepackage[
  bookmarksnumbered = true,
  hypertexnames = false,
  colorlinks    = true,
  citecolor     = gray,
  linkcolor     = gray,
  urlcolor      = gray,
  breaklinks
]{hyperref}

\DeclareFieldFormat{url}{%
  \href{#1}{\ComputerMouse}
}
\DeclareFieldFormat{doi}{%
  \mkbibacro{DOI}\addcolon\space
  \href{https://doi.org/#1}{#1}
}
\makeatletter
\DeclareFieldFormat{arxiv}{%
  arXiv\addcolon\space
  \href{http://arxiv.org/\abx@arxivpath/#1}{#1}
}
\makeatother
\DeclareFieldFormat{mr}{%
  MR\addcolon\space
  \href{http://www.ams.org/mathscinet-getitem?mr=MR#1}{#1}
}
\DeclareFieldFormat{zbl}{%
  Zbl\addcolon\space
  \href{http://zbmath.org/?q=an:#1}{#1}
}
\renewbibmacro*{eprint}{%
  \printfield{arxiv}%
  \newunit\newblock
  \printfield{mr}%
  \newunit\newblock
  \printfield{zbl}%
  \newunit\newblock
  \iffieldundef{eprinttype}
  {\printfield{eprint}}
  {\printfield[eprint:\strfield{eprinttype}]{eprint}}
}

\AtEveryBibitem{%
  \clearlist{publisher}%
}
\AtEveryBibitem{%
  \clearlist{address}%
}
\DeclareFieldFormat[article,inproceedings,inbook,incollection,thesis]{title}{\textit{#1}}
\renewbibmacro{in:}{}
\newcommand{\printreferences}{\printbibliography[heading=bibintoc]}

\usepackage[inline,shortlabels]{enumitem}

\usepackage{subcaption}

\usepackage[yyyymmdd]{datetime}

\usepackage{etoolbox}
\ifundef{\abstract}{}{\patchcmd{\abstract}%
    {\quotation}{\quotation\noindent\ignorespaces}{}{}}

\usepackage[super]{nth}

\input{preamble/envs.tex}
\input{preamble/macros.tex}
\input{preamble/letters.tex}


%% file: preamble/envs.tex
\usepackage{thmtools}

\numberwithin{equation}{section}

\renewcommand{\qedsymbol}{$\blacksquare$}

\newcommand{\CorollaryQED}{\qedsymbol}
\newcommand{\ConjectureQED}{$\square$}
\newcommand{\SituationQED}{$\times$}
\newcommand{\DefinitionQED}{$\bullet$}
\newcommand{\NotationQED}{$\circ$}
\newcommand{\ExampleQED}{$\spadesuit$}
\newcommand{\RemarkQED}{$\clubsuit$}

\declaretheorem[numberlike=equation,]{theorem}
\declaretheorem[numbered=no,name=Theorem]{theorem*}
\declaretheorem[numberlike=equation,name=Lemma]{lemma}
\declaretheorem[numberlike=equation,name=Proposition]{prop}
\declaretheorem[numberlike=equation,name=Corollary,qed=\CorollaryQED]{cor}

\declaretheorem[numberlike=equation,name=Hypothesis]{hypothesis}

\declaretheorem[numberlike=equation,name=Definition,style=definition,qed=\DefinitionQED]{definition}
\declaretheorem[numbered=no,name=Definition,style=definition,qed=\DefinitionQED]{definition*}

\declaretheorem[numberlike=equation,name=Data,style=definition]{data}
\declaretheorem[numberlike=equation,style=definition,qed=\ExampleQED]{example}

\declaretheorem[numberlike=equation,style=remark,qed=\RemarkQED]{remark}

\def\makeautorefname#1#2{\AtBeginDocument{\expandafter\def\csname#1autorefname\endcsname{#2}}}
\makeautorefname{table}{Table}        
\makeautorefname{chapter}{Chapter}
\makeautorefname{section}{Section}
\makeautorefname{subsection}{Section}
\makeautorefname{subsubsection}{Section}
\makeautorefname{footnote}{Footnote}
\AtBeginDocument{\def\itemautorefname~#1\null{(#1)\null}}
\AtBeginDocument{\def\equationautorefname~#1\null{(#1)\null}}

\numberwithin{substep}{step}
\makeautorefname{step}{Step}
\makeautorefname{substep}{Step}

\makeautorefname{case}{Case}
\makeautorefname{substep}{Step}

\setlist[description]{leftmargin=!,labelindent=1em}
\setlist[enumerate]{label={\rm (\arabic*)},ref=\arabic*}
\setlist[enumerate,2]{label={\rm (\alph*)},ref=\theenumi.\alph*}

%% file: preamble/macros.tex
\let\C\undefined
\let\U\undefined

\usepackage{bm}
\usepackage{mathtools} 
\usepackage{stmaryrd} 

\DeclareFontFamily{U}{mathx}{\hyphenchar\font45}
\DeclareFontShape{U}{mathx}{m}{n}{
      <5> <6> <7> <8> <9> <10>
      <10.95> <12> <14.4> <17.28> <20.74> <24.88>
      mathx10
      }{}
\DeclareSymbolFont{mathx}{U}{mathx}{m}{n}
\DeclareFontSubstitution{U}{mathx}{m}{n}
\DeclareMathAccent{\widecheck}{0}{mathx}{"71}
\DeclareMathAccent{\wideparen}{0}{mathx}{"75}

\DeclareMathOperator{\Ad}{Ad}

\DeclareMathOperator{\End}{End}

\DeclareMathOperator{\GL}{GL}

\DeclareMathOperator{\HF}{\HF}
\DeclareMathOperator{\Hess}{Hess}

\DeclareMathOperator{\Hom}{Hom}

\DeclareMathOperator{\Lie}{Lie}

\DeclareMathOperator{\Sym}{Sym}

\DeclareMathOperator{\im}{im}

\DeclareMathOperator{\tr}{tr}

\DeclarePairedDelimiter\floor{\lfloor}{\rfloor}
\DeclarePairedDelimiter\paren{\lparen}{\rparen}

\DeclarePairedDelimiter{\Abs}{\|}{\|}

\DeclarePairedDelimiter{\abs}{\lvert}{\rvert}
\DeclarePairedDelimiter{\bracket}{\langle}{\rangle}

\DeclarePairedDelimiter{\set}{\lbrace}{\rbrace}
\def\({\left(}
\def\){\right)}
\def\<{\left\langle}
\def\>{\right\rangle}

\newcommand{\C}{{\mathbf{C}}}
\newcommand{\Gtwo}{G_2}

\newcommand{\N}{{\mathbf{N}}}

\newcommand{\PSL}{\P\SL}

\newcommand{\R}{\mathbf{R}}

\newcommand{\SL}{\mathrm{SL}}
\newcommand{\SO}{\mathrm{SO}}
\newcommand{\SU}{\mathrm{SU}}

\newcommand{\Span}[1]{\bracket{#1}}
\newcommand{\Spin}{\mathrm{Spin}}
\newcommand{\Sp}{\mathrm{Sp}}

\newcommand{\U}{\mathrm{U}}

\newcommand{\Z}{\mathbf{Z}}

\newcommand{\co}{\mskip0.5mu\colon\thinspace}

\newcommand{\defined}[2][\key]{\def\key{#2}\textbf{#2}\index{#1}}

\newcommand{\del}{\partial}

\newcommand{\hkred}{{/\!\! /\!\! /}}

\newcommand{\id}{\mathrm{id}}

\newcommand{\inner}[2]{\bracket{#1, #2}}

\newcommand{\iso}{\cong}

\newcommand{\loc}{\mathrm{loc}}
\newcommand{\mathsc}[1]{{\normalfont\textsc{#1}}}
\newcommand{\nsub}{\triangleleft}

\newcommand{\one}{\mathbf{1}}

\newcommand{\qandq}{\quad\text{and}\quad}

\newcommand{\qforq}{\quad\text{for}\quad}
\newcommand{\qand}{\quad\text{and}}

\newcommand{\qwithq}{\quad\text{with}\quad}

\newcommand{\scN}{{\mathsc{n}}}

\newcommand{\su}{\mathfrak{su}}

\newcommand{\vol}{\mathrm{vol}}

\renewcommand{\H}{\mathbf{H}}
\renewcommand{\Im}{\operatorname{Im}}

\renewcommand{\P}{\mathbf{P}}

\renewcommand{\det}{\operatorname{det}}

\renewcommand{\epsilon}{\varepsilon}
\renewcommand{\setminus}{{\backslash}}
\renewcommand{\sp}{\mathfrak{sp}}

\renewcommand{\leq}{\leqslant}
\renewcommand{\geq}{\geqslant}

\newcommand{\Wedge}{\Lambda}

\makeatletter
\renewcommand*\env@matrix[1][*\c@MaxMatrixCols c]{%
  \hskip -\arraycolsep
  \let\@ifnextchar\new@ifnextchar
  \array{#1}}

\renewcommand\xleftrightarrow[2][]{%
  \ext@arrow 9999{\longleftrightarrowfill@}{#1}{#2}}
\newcommand\longleftrightarrowfill@{%
  \arrowfill@\leftarrow\relbar\rightarrow}
\makeatother


%% file: preamble/letters.tex
\newcommand{\rd}{{\rm d}}

\newcommand{\rII}{{\rm II}}

\newcommand{\bH}{{\mathbf{H}}}

\newcommand{\bS}{{\mathbf{S}}}

\newcommand{\sA}{\mathscr{A}}

\newcommand{\sH}{\mathscr{H}}

\newcommand{\ff}{{\mathfrak f}}
\newcommand{\fg}{{\mathfrak g}}

\newcommand{\fh}{{\mathfrak h}}

\newcommand{\fl}{{\mathfrak l}}

\newcommand{\fr}{{\mathfrak r}}
\newcommand{\fs}{{\mathfrak s}}
\newcommand{\ft}{{\mathfrak t}}
\newcommand{\fu}{{\mathfrak u}}

\newcommand{\fw}{{\mathfrak w}}

\newcommand{\fK}{{\mathfrak K}}

\newcommand{\fR}{{\mathfrak R}}

\newcommand{\slD}{\slashed D}

\newcommand{\bgamma}{\bm{\gamma}}

\newcommand{\bxi}{{\bm\xi}}
\newcommand{\bzeta}{{\bm\zeta}}

\newcommand{\ubR}{{\underline \R}}

%% file: Introduction.tex
\section{Introduction}

The study of the compactness problem for generalized Seiberg--Witten equations was pioneered by \citet{Taubes2012} with his compactness theorem for stable flat $\PSL_2(\C)$--connections in dimension three.
Building on the ideas developed in \cite{Taubes2012},
\citet{Haydys2014} proved a compactness theorem for the Seiberg--Witten equation with multiple spinors in dimension three, and
Taubes proved compactness theorems for the Kapustin--Witten equation \cite{Taubes2013},
the Vafa--Witten equation \cite{Taubes2017},
and the  Seiberg--Witten equation with multiple spinors in dimension four \cite{Taubes2016}.
Although the statements of these compactness theorems are very similar,
many details of their proofs seem to rely heavily on the particular structure of the equation under consideration.
The purpose of this article is to prove an abstract compactness theorem for generalized Seiberg--Witten equations in dimension three for which a simple analytical hypothesis holds.
Our result recovers Taubes' compactness theorem for stable flat $\PSL_2(\C)$--connections \cite{Taubes2012} as well as the compactness theorem for Seiberg--Witten equations with multiple spinors \cite{Haydys2014}.
Furthermore, it also implies a compactness theorem for the ADHM$_{1,2}$ Seiberg--Witten equation, which partially verifies a conjecture by \citet[Conjecture 5.26]{Doan2017d}.

\subsection{Generalized Seiberg--Witten equations}

Let us review the relation between quaternionic representations and generalized Seiberg--Witten equations on an oriented Riemannian $3$--manifold.
For more detailed discussions we refer the reader to \cites{Taubes1999b}{Haydys2013}{Doan2017a}[Appendix B]{Doan2017d}.

\begin{definition}
  Denote by $\bH = \R\Span{1,i,j,k}$ the normed division algebra of the quaternions.
  A \defined{quaternionic Hermitian vector space} is a left $\bH$--module $S$ together with an Euclidean inner product $\inner{\cdot}{\cdot}$ such that $i,j,k$ act by isometries.
  The \defined{unitary symplectic group} $\Sp(S)$ is the subgroup of $\GL_\bH(S)$ preserving $\inner{\cdot}{\cdot}$.
\end{definition}

\begin{definition}
  A \defined{quaternionic representation} of a Lie group $H$ is a Lie group homomorphism $\rho\co H \to \Sp(S)$ for some quaternionic Hermitian vector space $S$.
\end{definition}

Let $H$ be a compact Lie group.
Denote its Lie algebra by $\fh$.
Let $\rho\co H \to \Sp(S)$ be a quaternionic representation.
Abusing notation, we denote the induced Lie algebra representation by $\rho\co \fh \to \sp(S)$.
Define $\gamma \co \Im\bH \to \End(S)$,
$\bgamma\co \Im\bH\otimes\fh \to \End(S)$, and
$\mu\co S \to (\Im\H\otimes\fh)^*$ by  
\begin{equation}
  \label{Eq_PointwiseMaps_Definition}
  \gamma(v)\phi
  \coloneq
  v\phi, \quad
  \bgamma(v\otimes\xi) \coloneq \gamma(v)\rho(\xi), \qandq
  \mu(\phi) \coloneq \frac12 \bgamma^*(\phi\phi^*),
\end{equation}
respectively.
The map $\mu$ is an equivariant hyperkähler moment map for the action of $H$ on $S$.

\begin{data}
  A set of \defined{algebraic data} consists of:
  \begin{enumerate}
  \item
    a compact Lie group $H$ with a distinguished element $-1 \in Z(H)$ satisfying $(-1)^2 = 1_H$,
  \item
    a closed, connected, normal subgroup $G \nsub H$, and
  \item
    a quaternionic representation $\rho\co H \to \Sp(S)$.
  \end{enumerate}
\end{data}

\begin{remark}
  It is $G$ which plays role of the structure group of the gauge theory.
  If $G$ is a proper subgroup of $H$,
  then the gauge theory can be twisted by the \defined{flavor symmetry group}
  \begin{equation*}
    K \coloneq H/\Span{G,-1}.
    \qedhere
  \end{equation*}
\end{remark}

\begin{definition}
  Set $\Spin^H(3) \coloneq (\Sp(1) \times H)/\Z_2$.
  $\Spin^H(3)$ projects onto $\Sp(1)/\Z_2 = \SO(3)$.
  A \defined{spin$^H$ structure} on $(M,g)$ is a principal $\Spin^H(3)$--bundle $\fs$ together with an isomorphism
  \begin{equation*}
    \fs \times_{\Spin^H(3)} \SO(3) \iso \SO(TM).
    \qedhere
  \end{equation*}
\end{definition}

A spin$^H$ structure $\fs$ together with $G\nsub H$ and $\rho$ induces:
\begin{enumerate}
\item
  the \defined{flavor bundle}
  \begin{equation*}
    \ff \coloneq \fs\times_{\Spin^H(3)} K,
  \end{equation*}
\item
  the \defined{adjoint bundle}
  \begin{equation*}
    \Ad(\fs) \coloneq \fs \times_{\Spin^H(3)} \Lie(G),
  \end{equation*}
\item
  the \defined{spinor bundle}
  \begin{equation*}
    \bS \coloneq \fs\times_{\Spin^H(3)} S,
  \end{equation*}
  as well as
\item
  maps
  \begin{equation*}
    \gamma \co TM \to \End(S), \quad
    \bgamma \co TM\otimes\Ad(\fs) \to \End(S), \qandq
    \mu\co \bS \to \Wedge^2\Tstar M \otimes \Ad(\fs),
  \end{equation*}
  where $\gamma$ and $\bgamma$ are induced directly by \eqref{Eq_PointwiseMaps_Definition}, and
  $\mu$ is induced by \eqref{Eq_PointwiseMaps_Definition} and the isomorphism $\Wedge^2\Tstar M \otimes \Ad(\fs) \iso \Tstar M \otimes \Ad(\fs)^*$.
\end{enumerate}

\begin{definition}
  A \defined{spin connection} on $\fs$ is a connection which induces the Levi--Civita connection on $TM$.
  The space of all spin connections on $\fs$ inducing a fixed connection $B$ on the flavor bundle $\ff$ is denoted by
  \begin{equation*}
    \sA(\fs,B).
  \end{equation*}
  Given a spin connection $A$,
  denote by
  \begin{equation*}
    \Ad(A) \in \sA(\Ad(\fs))
  \end{equation*}
  the induced connection on $\Ad(\fs)$
  and define the \defined{Dirac operator} $\slD_A \co \Gamma(\bS) \to \Gamma(\bS)$ by
  \begin{equation*}
    \slD_A\Phi \coloneq \sum_{i=1}^3 \gamma(e_i)\nabla_{A,e_i}\Phi
  \end{equation*}
  for $e_1,e_2,e_3$ a local orthonormal frame.
\end{definition}

\begin{data}
  A set of \defined{geometric data} compatible with a given set of algebraic data $(G,H,\rho)$ consists of:
  \begin{enumerate}
  \item an oriented Riemannian $3$--manifold $(M,g)$ together with a spin$^H$ structure $\fs$, and
  \item a connection $B$ on the flavor bundle induced by $\fs$.
  \end{enumerate}
\end{data}

\begin{definition}
  The \defined{generalized Seiberg--Witten equation} associated with the data $(G,H,\rho)$ and $(M,g,\fs,B)$ is the following partial differential equation for $A \in \sA(\fs,B)$ and $\Phi \in \Gamma(\bS)$:
  \begin{equation}
    \label{Eq_SeibergWitten}
    \slD_A\Phi = 0 \qandq
    F_{\Ad(A)} = \mu(\Phi).                 
    \qedhere
  \end{equation}
\end{definition}

To illustrate the above construction,
let us consider a few examples.

\begin{example}
  \label{Ex_ClassicalSeibergWitten}
  Define the quaternionic representation $\rho\co \U(1) \to \Sp(\H)$ by
  \begin{equation*}
    \rho(e^{i\alpha})q \coloneq q e^{i\alpha}.
  \end{equation*}  
  Identifying $(i\R\otimes \Im \H)^* = i\R\otimes\Im \H$,
  the hyperkähler moment map $\mu\co \H \to (i\R\otimes \Im \H)^*$ is
  \begin{equation*}
    \mu(q) = -\frac{i}{2} \otimes qiq^*.
  \end{equation*}
  Splitting $\H = \C \oplus j\C$,  
  we see that $\bgamma(\mu(q)) \in \End(\C^{\oplus 2})$ for $q = z + jw$ is
  \begin{equation}
    \label{Eq_ClassicalSeibergWitten_MomentMap}
    \frac12
    \begin{pmatrix}
      \abs{z}^2 - \abs{w}^2 & 2z\bar w \\
      2 \bar z w & \abs{w}^2 - \abs{z}^2
    \end{pmatrix}
    =
    q\inner{q}{\cdot}_\C - \frac12\abs{q}_\C^2\,\id_{\C^{\oplus 2}}.
  \end{equation}
  
  Let $(M,g)$ be an oriented Riemannian $3$--manifold and let $\fs$ be a spin$^{\U(1)}$ structure on $M$;
  that is: a spin$^c$ structure.
  The adjoint bundle $\Ad(\fs)$ is $i\ubR$.
  Denote the spinor bundles of $\fs$ by $\bS$.  
  If $A \in \sA(\fs)$,
  then it induces a connection $\det(A)$ on $\det(\bS)$ with
  \begin{equation*}
    F_{\det(A)} = 2F_{\Ad(A)}.
  \end{equation*}
  Therefore,
  the generalized Seiberg--Witten equation \autoref{Eq_SeibergWitten} associated with the above data agrees with the \defined{classical Seiberg--Witten equation}
  \begin{align*}
    \slD_A \Phi &= 0 \qand \\
    \frac12 \bgamma(F_{\det(A)}) &= \Phi\inner{\Phi}{\cdot}_\C - \abs{\Phi}_\C^2\id_\bS
  \end{align*}
  appearing, for example, in \cites[Section 2]{Witten1994}[Section 1.3]{Kronheimer2007}.  
\end{example}

\begin{example}
  \label{Ex_GCFlatness}
  Let $G$ be a compact Lie group and set $\fg \coloneq \Lie(G)$.
  Choosing a $G$--invariant inner product on $\fg$ turns $S \coloneq \fg\otimes_\R\H$ into a quaternionic Hermitian vector space.
  The adjoint representation induces a quaternionic representation $\rho \co G \to \Sp(S)$.
  The moment map $\mu\co S \to \Im\H\otimes\fg$ is given by
  \begin{align*}
    \mu(\bxi)
    &=
      \frac12[\bxi,\bxi] \\
    &=
      ([\xi_2,\xi_3]+[\xi_0,\xi_1])\otimes i
    + ([\xi_3,\xi_1]+[\xi_0,\xi_2])\otimes j
    + ([\xi_1,\xi_2]+[\xi_0,\xi_3])\otimes k
  \end{align*}
  for $\bxi = \xi_0\otimes 1 + \xi_1\otimes i + \xi_2\otimes j + \xi_3\otimes k \in \H\otimes\fg$.
  Extend $\rho$ to a quaternionic representation of $H \coloneq \Sp(1) \times G$ by declaring that $q \in \Sp(1)$ acts by right-multiplication with $q^*$.
  Set $-1 \coloneq (-\one,1_G) \in H$.
  
  Since
  \begin{equation*}
    \Spin^H(3) = (\Sp(1)\times\Sp(1))/\Z_2\times G =\SO(4)\times G,
  \end{equation*}
  a spin$^H$ structure is nothing but an oriented Euclidean vector bundle $N$ of rank $4$ together with an orientation-preserving isometry $\Wedge^+ N \iso TM$ and a principal $G$--bundle $P$.
  Choosing $N = \ubR\oplus \Tstar M$ and $B$ induced by the Levi-Civita connection,
  the generalized Seiberg--Witten equation \eqref{Eq_SeibergWitten} associated with the above data becomes the following partial differential equation for $A \in \sA(P)$, $a \in \Omega^1(M,\Ad(P))$, and $\xi \in \Gamma(\Ad(P))$:
  \begin{equation}
    \label{Eq_StableFlatGC}
    \begin{split}
      \rd_A^*a &= 0, \\
      *\rd_Aa + \rd_A\xi &= 0, \qand \\
      F_A &= \tfrac12[a\wedge a] + *[\xi,a].      
    \end{split}
  \end{equation}

  If $\xi = 0$, then \autoref{Eq_StableFlatGC} is precisely the condition for $A + ia$ to be a \defined{stable flat $G^\C$--connection};
  see \cites{Donaldson1987}[Theorem 3.3]{Corlette1988}.
  In fact, if $M$ is closed,
  then \autoref{Eq_StableFlatGC} implies $\rd_A\xi = 0$ and $[\xi,a] = 0$ and,
  therefore, that $A+ia$ is a stable flat $G^\C$--connection.

  The compactness problem for \autoref{Eq_StableFlatGC} with $G = \SO(3)$ has been considered in \citeauthor{Taubes2012}' pioneering work \cite{Taubes2012},
  to which many of the techniques in this article can be traced back.
\end{example}

\begin{example}
  \label{Ex_ADHMSeibergWitten}
  For $r,k \in \N$,
  consider the quaternionic Hermitian vector space
  \begin{equation*}
    S_{r,k} \coloneq \Hom_\C(\C^r,\H\otimes_\C \C^k) \oplus \H^*\otimes_\R\fu(k)
  \end{equation*}
  and
  \begin{equation*}
    G = \U(k) \nsub H = \SU(r) \times \Sp(1) \times \U(k) \qandq
    -1 \coloneq (\one,-\one,-\one).
  \end{equation*}
  If $r \geq 2$,
  then $S_{r,k}\hkred G \coloneq \mu^{-1}(0)/G$ is the Uhlenbeck compactification of the moduli space of framed $\SU(r)$ ASD instantons of charge $k$ on $\R^4$ \cite{Atiyah1978}.
  If $r = 1$, then
  \begin{equation*}
    S_{1,k}\hkred G = \Sym^k\H \coloneq \H^k\!/S_k;
  \end{equation*}
  see \cites[Proposition 2.9]{Nakajima1999}[Theorem D.2]{Doan2017d}.

  The generalized Seiberg--Witten equation associated with the above data is called the \defined{ADHM$_{r,k}$ Seiberg--Witten equation}.
  It was introduced in \cites[Example A.3]{Doan2017a}[Section 5.1]{Doan2017d} and is expected to play an important role in gauge theory on $\Gtwo$--manifolds \cite{Donaldson2009,Walpuski2013a,Haydys2017}.
  For $k = 1$,
  this is essentially the Seiberg--Witten equation with $r$ spinors,
  whose compactness problem has been considered by \citet{Haydys2014}.
\end{example}

\subsection{An abstract compactness theorem}

Throughout this subsection,
fix a set of algebraic data $(G,H,\rho)$ and a compatible set of geometric data $(M,g,\fs,B)$ with $M$ closed.
The following result is well-known and follows from standard elliptic theory.

\begin{prop}
  \label{Prop_CompactnessProvidedUniformL2ZBounds}
  If $(A_n,\Phi_n)$ is a sequence of solutions of \autoref{Eq_SeibergWitten} satisfying
  \begin{equation*}
    \liminf_{n\to\infty}\, \Abs{\Phi_n}_{L^2} < \infty,
  \end{equation*}
  then, after passing to a subsequence and up to gauge transformations,
  $(A_n,\Phi_n)$ converges to a solution $(A,\Phi)$ of \autoref{Eq_SeibergWitten} in the $C^\infty$ topology.
\end{prop}

Therefore,
a degenerating sequence $(A_n,\Phi_n)$ of solutions of \autoref{Eq_SeibergWitten} must involve $\Abs{\Phi_n}_{L^2}$ becoming unbounded.
In light of this, it is convenient to pass to the following equivalent equation.

\begin{definition}
  The \defined{blown-up generalized Seiberg--Witten equation} associated with the data $(G,H,\rho)$ and $(M,g,\fs,B)$ is the following partial differential equation for
  $A \in \sA(\fs,B)$,
  $\Phi \in \Gamma(\bS)$, and
  $\epsilon \in (0,\infty)$:
  \begin{gather}
    \label{Eq_UnnormalizedBlownupSeibergWitten}
    \slD_A\Phi = 0, \quad
    \epsilon^2 F_{\Ad(A)} = \mu(\Phi), \qand  \\
    \label{Eq_NormalizationOfTheSpinor}
    \Abs{\Phi}_{L^2} = 1.
    \qedhere
  \end{gather}
\end{definition}

The main result of this article is the following abstract compactness theorem.

\begin{definition}
  Given $\Phi \in \bS$,
  define $\Gamma_\Phi\co \Wedge^2\Tstar M \to \bS$ by
  \begin{equation*}
    \Gamma_\Phi \coloneq \bgamma(\cdot)\Phi.
    \qedhere
  \end{equation*}
\end{definition}

\begin{remark}
  $\Gamma_\Phi$ is one-half times the adjoint of $\rd_\Phi\mu$.
\end{remark}

\begin{hypothesis}
  \label{Hyp_GammaFControlsF}
  There are constants $r_0,\delta_\mu,c > 0$ and $\Lambda \geq 0$ such that the following holds for every $x \in M$ and $r \in (0,r_0]$.
  If $A \in \sA(\fs,B)$,
  $\Phi \in \Gamma(\bS)$, and
  $\epsilon \in (0,\infty)$
  satisfy \autoref{Eq_UnnormalizedBlownupSeibergWitten},
  \begin{equation*}
    \frac12 \leq \abs{\Phi} \leq 2, \qandq
    \abs{\mu(\Phi)} \leq \delta_\mu,
  \end{equation*}
  on $B_r(x)$,
  then
  \begin{equation}
  \label{Eq_Hypothesis}
    r\int_{B_{r/2}(x)} \abs{F_{\Ad(A)}}^2
    \leq
      \Lambda + cr\int_{B_r(x)} \abs{\Gamma_\Phi F_{\Ad(A)}}^2.
  \end{equation}
\end{hypothesis}

\begin{remark}
  \label{Rmk_GeometricCriterion}  
  \autoref{Hyp_GammaFControlsF} (with $\Lambda = 0$) is implied by the following condition:
  there are constants $\delta,c > 0$ such that,
  for every $\Phi \in \bS$ with $\abs{\Phi} = 1$ and $\abs{\mu(\Phi)} \leq \delta$,
  \begin{equation}
    \label{Eq_GeometricCriterion}
    \abs{\mu(\Phi)} \leq c\abs{\Gamma_\Phi\mu(\Phi)}.
    \qedhere
  \end{equation}
\end{remark}

\begin{remark}
  The condition in \autoref{Rmk_GeometricCriterion} holds if $\mu^{-1}(0)$ is cut-out transversely away from the origin;
  that is: for every non-zero $\Phi \in \mu^{-1}(0)$, $\rd_\Phi\mu$ is surjective and, therefore, $\Gamma_\Phi$ is injective.
  This is the case for the quaternionic representation $\U(1) \to \Sp(\H^n)$ which induces the Seiberg--Witten equation with multiple spinors.
  Therefore, \autoref{Thm_AbstractCompactness} recovers \cite[Theorem 1.5]{Haydys2014}.
\end{remark}

\begin{remark}
  For the of the adjoint representation $G \to \Sp(\fg\otimes\H)$,
  $\mu^{-1}(0)$ is never cut-out transversely away from the origin.
  Nevertheless, \autoref{Lem_GammaXiMuControlsMu} shows that the algebraic criterion in \autoref{Rmk_GeometricCriterion} is satisfied for $G = \SO(3)$ and $G = \SU(2)$.
  Therefore, \autoref{Thm_AbstractCompactness} applies to stable flat $\PSL_2(\C)$--connections over $3$--manifolds;
  cf. \autoref{Rmk_PSL2CCompactness}.
\end{remark}

\begin{theorem}
  \label{Thm_AbstractCompactness}
  Suppose \autoref{Hyp_GammaFControlsF} holds.
  If $(A_n,\Phi_n,\epsilon_n)_{n \in \N}$ is a sequence of solutions of \autoref{Eq_UnnormalizedBlownupSeibergWitten} and \autoref{Eq_NormalizationOfTheSpinor} with $\epsilon_n$ tending to zero,
  then the following hold:
  \begin{enumerate}
  \item
    There is a closed, nowhere-dense subset $Z \subset M$,
    a connection $A \in \sA(\fs|_{M\setminus Z},B)$, and
    a spinor $\Phi \in \Gamma(M\setminus Z,\bS)$ such that the following hold:
    \begin{enumerate}
    \item
      $A$ and $\Phi$ satisfy
      \begin{equation}
        \label{Eq_LimitingSeibergWitten}
        \begin{split}
          \slD_A\Phi
          &= 0, \\
          \mu(\Phi)
          &=
          0, \qand \\
          \Abs{\Phi}_{L^2}
          &=
          1.
        \end{split}
      \end{equation}
    \item
      The function $\abs{\Phi}$ extends to a Hölder continuous function on all of $M$ and
      \begin{equation*}
        Z = \abs{\Phi}^{-1}(0).
      \end{equation*}
    \end{enumerate}
  \item
    After passing to a subsequence and up to gauge transformations,
    for every compact subset $K\subset M\setminus Z$,
    $\paren{A_n|_K}_{n \in \N}$ converges to $A$ in the weak $W^{1,2}$ topology,
    $\paren{\Phi_n|_K}_{n \in \N}$ converges to $\Phi$ in the weak $W^{2,2}$ topology,
    and there exists an $\alpha \in (0,1)$ such that $(\abs{\Phi_n})_{n \in \N}$ converges to $\abs{\Phi}$ in the $C^{0,\alpha}$ topology.
  \end{enumerate}
\end{theorem}

\begin{remark}
  For many generalized Seiberg--Witten equations,
  including the Seiberg--Witten equation with multiple spinors and stable flat $\PSL_2(\C)$--connections,
  a solution of \autoref{Eq_LimitingSeibergWitten} gives rise to a harmonic $\Z_2$ spinor whose zero locus is precisely $Z$;
  cf. \cite{Taubes2014} and \autoref{Sec_ConclusionOfProofOfADHM12SeibergWittenCompactness}.
  In this case, \citet[Theorem 1.4]{Zhang2017} proved that $Z$ is $\sH^1$--rectifiable
  and and has finite $1$--dimensional Minkowski content.%
  \footnote{%
    See \cites[Chapter 15]{Mattila1995}[Chapter 4]{DeLellis2008} for discussions of rectifiability.
    A subset $Z \subset M$ is said to have finite $1$--dimensional Minkowski content if there is a constant $c > 0$ such that $\vol(\set{ x \in M : d(x,Z) < r }) \leq cr$.
    }
\end{remark}

\subsection{A compactness theorem for the ADHM\texorpdfstring{$_{1,2}$}{12} Seiberg--Witten equation}
\label{Sec_ADHM12SeibergWittenCompactness}

Let us discuss \autoref{Ex_ADHMSeibergWitten} for $r=1$ and $k=2$ in more detail.
Decomposing $\fu(2) = \su(2)\oplus\fu(1)$,
$S = S_{1,2}$ can be written as
\begin{equation*}
  S = S_\circ \oplus \H\otimes_\R\fu(1)
  \qwithq
  S_\circ \coloneq \H\otimes_\C \C^2 \oplus \H\otimes_\R \su(2).
\end{equation*}
$\U(2)$ acts trivially on $\H\otimes_\R\fu(1)$;
hence, the moment map $\mu\co S \to \fu(2) \otimes \Im\H$ factors through the projection of $S$ onto $S_\circ$.

Let $(M,g)$ be a closed Riemannian $3$--manifold.
A spin$^{\Sp(1)\times\U(2)}$ structure $\fs$ on $(M,g)$ is nothing but a spin$^{\U(2)}$ structure $\fw$ and a Euclidean vector bundle $N$ of rank $4$ together with an orientation-preserving isometry
\begin{equation*}
  \Wedge^+N \iso TM.
\end{equation*}
Set
\begin{equation*}
  W \coloneq \fw\times_{\Spin^{\U(2)}(3)} \H\otimes_\C\C^2 \qandq
  \Ad(\fw)_\circ \coloneq \fw\times_{\Spin^{\U(2)}(3)} \su(2).
\end{equation*}
The spinor bundle $\bS$ and the flavor bundle $\ff$ associated with $\fs$ are
\begin{equation*}
  \bS = W \oplus N\otimes \Ad(\fw)_{\circ} \oplus N\otimes i\ubR \qandq
  \ff = \SO(\Wedge^-N).
\end{equation*}
Given a connection $B$ on $\SO(\Lambda^-N)$,
every connection on $\Ad(\fw)$ uniquely lifts to a spin connection on $\fs$.

The above discussion shows that,
having fixed $B$,
the ADHM$_{1,2}$ Seiberg--Witten equation is the following partial differential equation for $A \in \sA(\Ad(\fw))$, $\Psi \in \Gamma(W)$, and $\bxi \in \Gamma(N\otimes\Ad(\fw)_\circ)$:
\begin{equation}
  \label{Eq_ADHM12SeibergWitten}
  \begin{split}
    \slD_A\Psi
    &=
    0 \\
    \slD_{A,B}\bxi
    &=
    0, \qand \\
    F_A
    &=
    \mu(\Psi,\bxi),
  \end{split}  
\end{equation}
as well as the Dirac equation for $\eta \in \Gamma(N\otimes i\ubR)$:
\begin{equation}
  \label{Eq_Dirac}
  \slD_B \eta = 0.
\end{equation}
The equations \autoref{Eq_ADHM12SeibergWitten} and \autoref{Eq_Dirac} are completely decoupled.
The compactness problem for \autoref{Eq_Dirac} is trivial:
after renormalization every sequence has a subsequence which converges in the $C^\infty$ topology.
Of course,
\autoref{Prop_CompactnessProvidedUniformL2ZBounds} applies to the ADHM$_{1,2}$ Seiberg--Witten equation \autoref{Eq_ADHM12SeibergWitten}.
The following result concerns the case in which the hypothesis of  \autoref{Prop_CompactnessProvidedUniformL2ZBounds} is not satisfied.

\begin{theorem}
  \label{Thm_ADHM12SeibergWittenCompactness}
  If $(A_n,\Psi_n,\bxi_n)_{n \in \N}$ is a sequence of solutions of \autoref{Eq_ADHM12SeibergWitten} with
  \begin{equation*}
    \liminf_{n\to\infty}\, \Abs{(\Psi_n,\bxi_n)}_{L^2} = \infty,
  \end{equation*}
  then the following hold:
  \begin{enumerate}
  \item
    There is a closed $\sH^1$--rectifiable subset $Z \subset M$ with finite $1$--dimensional Minkowski content,
    a connection $A \in \sA(\fw|_{M\setminus Z})$,
    a spinor $\Psi \in \Gamma(M\setminus Z,W)$,
    a section $\bxi \in \Gamma(M\setminus Z,N\otimes\Ad(\fw)_\circ)$,
    a flat Euclidean line bundle $\fl$ over $M\setminus Z$, and
    a non-zero $\tau \in \Gamma(M\setminus Z,\Hom(\fl,\Ad(\fw)_\circ))$ such that the following hold:
    \begin{enumerate}
    \item
      $A$ and $\bxi$ satisfy
      \begin{equation}
        \label{Eq_LimitingADHM12SeibergWitten_Xi}
        \begin{split}
          \slD_{A,B}\bxi
          &= 0, \\
          \mu(\bxi)
          &=
          0, \qand \\
          \Abs{\bxi}_{L^2}
          &=
          1.
        \end{split}
      \end{equation}      
    \item
      The function $\abs{\bxi}$ extends to a Hölder continuous function on all of $M$ and
      \begin{equation*}
        Z = \abs{\bxi}^{-1}(0).
      \end{equation*}
    \item
      The section $\tau$ is parallel with respect to $A$.
    \item
      Set $\ft \coloneq \im \tau \oplus i\ubR \subset \Ad(\fw)|_{M\setminus Z}$ and denote by $\pi_\ft\co \Ad(\fw)|_{M\setminus Z} \to \ft$ the orthogonal projection onto $\ft$.
      $A$ and $\Psi$ satisfy
      \begin{equation}
        \label{Eq_LimitingADHM12SeibergWitten_Psi}
        \begin{split}
          \slD_A\Psi
          &= 0 \qand \\
          F_A
          &=
          \pi_\ft\mu(\Psi).
        \end{split}
      \end{equation}
    \end{enumerate}    
  \item
    Set
    \begin{equation*}
      \epsilon_n
      \coloneq
        \frac{1}{\Abs{(\Psi_n,\bxi_n)}_{L^2}}, \quad
      \tilde\Psi_n
      \coloneq
        \epsilon_n\Psi_n, \qandq
      \tilde\bxi_n
      \coloneq
        \epsilon_n\bxi_n.
    \end{equation*}
    After passing to a subsequence and up to gauge transformations,
    for every compact subset $K \subset M\setminus Z$,
    $\paren{A_n|_K}_{n \in \N}$ converges to $A$ in the weak $W^{1,2}$ topology,
    $\paren{\Psi_n|_K}_{n \in \N}$ converges to $\Psi$ in the weak $W^{2,2}$ topology,
    $\paren{\tilde\bxi_n|_K}_{n \in \N}$ converges to $\bxi$ in the $W_\loc^{2,2}$ topology, and
    there exists an $\alpha \in (0,1)$ such that $\paren{\abs{(\tilde\Psi_m,\tilde\bxi_n)}}_{n \in \N}$ converges to $\abs{\bxi}$ in the $C^{0,\alpha}$ topology.   
  \end{enumerate}
\end{theorem}

\begin{remark}
  We emphasize that $(\Psi_n)_{n \in \N}$ converges \emph{without rescaling}.
\end{remark}

\begin{remark}
  \label{Rmk_PSL2CCompactness}
  \autoref{Thm_ADHM12SeibergWittenCompactness} with $\Psi_n = 0$ recovers Taubes' compactness theorem for stable flat $\PSL_2(\C)$--connections over $3$--manifolds \cite{Taubes2012}.
  In fact, it also shows that the limiting connection $A$ is flat.
\end{remark}

\begin{remark}
  \autoref{Thm_ADHM12SeibergWittenCompactness} partially verifies \cite[Conjecture 5.26]{Doan2017d}.
  It is explained in \cite[Section 5]{Doan2017d} that:
  the joint spectrum of the section $\bxi$ provides a double cover $\pi\co \tilde M\setminus\tilde Z \to M\setminus Z$, which extends to a branched cover of $M$;
  the spin$^{\U(2)}$ structure on $M\setminus Z$ is the push-forward of a spin$^c$ structure on $\tilde M\setminus\tilde Z$;
  and is $(A,\Psi)$ the push-forward of a solution of the Seiberg--Witten equation on $\tilde M\setminus\tilde Z$.
  The interested reader is referred to \cite[Sections 4, 5, 6, and 7]{Doan2017d} for a discussion of how the ADHM$_{1,k}$ Seiberg--Witten equation is expected help in dealing with multiple cover phenomena for associatives in $\Gtwo$--manifolds and pseudo-holomorphic curves in symplectic Calabi--Yau $3$--folds.
\end{remark}

\paragraph{Conventions}
Throughout,
fix a set of algebraic data $(G,H,\rho)$ and a compatible set of geometric data $(M,g,\fs,B)$ with $M$ closed.
As is customary, $c > 0$ denotes a universal constant whose value might change from on appearance to the next and which depends only on the chosen algebraic and geometric data.
Moreover, $r_0 > 0$ denotes a constant which is much smaller than the injectivity radius and at least as small as the constant appearing in \autoref{Hyp_GammaFControlsF}.

\paragraph{Acknowledgements}

This material is based upon work supported by  \href{https://www.nsf.gov/awardsearch/showAward?AWD_ID=1754967&HistoricalAwards=false}{the National Science Foundation under Grant No.~1754967}
and
\href{https://sloan.org/grant-detail/8651}{an Alfred P. Sloan Research Fellowship}.


%% file: LichnerowiczWeitzenbock.tex
\section{The Lichnerowicz--Weitzenböck formula}
\label{Sec_LichnerowiczWeitzenbock}

This section derives a number of consequences of the Lichnerowicz--Weitzenböck formula.
Let us begin by reminding the reader of the latter.

\begin{definition}
  Define $\fR \in \Omega^2(M,\End(\bS))$ and $\fK \in \Gamma(\End(\bS))$ by
  \begin{align*}
    \fR\Phi
    \coloneq
    \frac14\sum_{i,j=1}^3 \inner{R(\cdot,\cdot)e_i}{e_j}\gamma(e_i)\gamma(e_j)\Phi
    + F_B
    \qandq
    \fK \Phi
    \coloneq
    \gamma(\fR)\Phi.
    \qedhere
  \end{align*}
\end{definition}

\begin{prop}[Lichnerowicz--Weitzenböck formula]
  \label{Prop_LichnerowiczWeitzenbock}
  For every $A \in \sA(\fs,B)$ and $\Phi \in \Gamma(\bS)$,
  \begin{equation}
    \label{Eq_LichnerowiczWeitzenbock}
    \slD_A^2\Phi
    = \nabla_A^*\nabla_A\Phi
    + \bgamma(F_{\Ad(A)})\Phi
    + \fK\Phi.
  \end{equation}
\end{prop}

If \autoref{Eq_UnnormalizedBlownupSeibergWitten} holds,
then \autoref{Prop_LichnerowiczWeitzenbock} implies
\begin{equation}
  \label{Eq_LaplacianOfAbsPhiSquaredEqualsLowerOrderTerm}
  \frac12\Delta\abs{\Phi}^2
  + \abs{\nabla_A\Phi}^2
  + 2\epsilon^{-2}\abs{\mu(\Phi)}^2
  + \inner{\fK\Phi}{\Phi}
  = 0.
\end{equation}
The following is an immediate consequence of \autoref{Eq_LaplacianOfAbsPhiSquaredEqualsLowerOrderTerm} and integration by parts.

\begin{cor}
  \label{Cor_IntegrationByPartsFormula}
  Let $U$ be an open subset of $M$ with smooth boundary and let $f \in C^\infty(\bar U)$.
  If $A \in \sA(\fs,B)$, $\Phi \in \Gamma(\bS)$, and $\epsilon > 0$ satisfy \autoref{Eq_UnnormalizedBlownupSeibergWitten} on $U$,
  then
  \begin{equation*}
    \int_U
      \frac12 \Delta f \cdot \abs{\Phi}^2
      + f \cdot \paren*{\abs{\nabla_A\Phi}^2 + 2\epsilon^{-2}\abs{\mu(\Phi)}^2}
      =
      - \int_U f\cdot \inner{\fK \Phi}{\Phi}
      + \frac12 \int_{\del U} f\cdot \del_\nu \abs{\Phi^2} - \del_\nu f \cdot \abs{\Phi}^2.
      \qedhere
  \end{equation*}
\end{cor}

\subsection{The frequency function}

The statements of the results derived in this section require the following definitions.

\begin{definition}
  \label{Def_Frequency}
  Given $A \in \sA(\fs)$, $\Phi \in \Gamma(\bS)$, $x \in M$, and $\epsilon > 0$,
  define $m_x^\Phi,D_x^{A,\Phi,\epsilon}\co (0,r_0] \to [0,\infty)$ by
  \begin{align*}
    m_x^\Phi(r)
    &\coloneq
      \frac{1}{4\pi r^2}
      \int_{\del B_r(x)} \abs{\Phi}^2 \qand \\
    D_x^{A,\Phi,\epsilon}(r)
    &\coloneq
      \frac{1}{4\pi r}
      \int_{B_r(x)} \abs{\nabla_A\Phi}^2 + 2\epsilon^{-2}\abs{\mu(\Phi)}^2;
  \end{align*}
  and, furthermore,
  set $r_{-1,x}^\Phi \coloneq \sup \set*{ r \in (0,\infty) : m_x^\Phi(r) = 0 }$ and
  define the \defined{frequency function} $\scN_x^{A,\Phi,\epsilon}\co (r_{-1,x}^\Phi,r_0] \to [0,\infty)$ by
  \begin{equation*}
    \scN_x^{A,\Phi,\epsilon}(r)
    \coloneq
    \frac{D_x^{A,\Phi,\epsilon}(r)}{m_x^\Phi(r)}.
    \qedhere
  \end{equation*}
\end{definition}

\begin{remark}
  A priori, the restriction of the domain of $\scN_x^{A,\Phi,\epsilon}$ is necessary;
  however: it be will shown in \autoref{Prop_M>0} that $r_{-1,x}^\Phi = 0$ unless $\Phi = 0$.
\end{remark}

\begin{remark}
  The frequency function was introduced by \citet{Almgren1979} and is now an ubiquitous tools in the study of elliptic partial differential equations.
  The adaption to generalized Seiberg--Witten equations is due to \citet{Taubes2012}.
\end{remark}

For the purposes of this section we shall be content with just the above definitions.
However, in \autoref{Sec_RegularityScale}, the frequency function plays a pivotal role and its properties will be studied in detail.

\subsection{\texorpdfstring{$L^2$}{L2} bounds on \texorpdfstring{$\Phi$}{Phi}}

\begin{prop}
  \label{Prop_L2Bounds}
  If $A \in \sA(\fs,B)$, $\Phi \in \Gamma(\bS)$, and $\epsilon > 0$ satisfy \autoref{Eq_UnnormalizedBlownupSeibergWitten},
  then, for every $x \in M$ and $r \in (0,r_0]$,
  \begin{equation*}    
    \frac{\pi}{2} m_x^\Phi\paren*{\frac r2}
    \leq 
      r^{-3}\int_{B_r(x)} \abs{\Phi}^2
    \leq
      4\pi m_x^\Phi(r).
  \end{equation*}
\end{prop}

\begin{proof}
  Denote by $H_{x,r}$ the mean curvature of $\del B_r(x)$.
  By \autoref{Cor_IntegrationByPartsFormula} with $f = 1$ and $U = B_r(x)$,
  \begin{align*}
    \frac{\rd}{\rd r} \int_{\del B_r(x)} \abs{\Phi}^2
    &=
      \int_{\del B_r(x)} H_{x,r} \abs{\Phi}^2
      + \int_{\del B_r(x)} \del_r\abs{\Phi}^2 \\
    &=
      \int_{\del B_r(x)} H_{x,r} \abs{\Phi}^2
      + 2\int_{B_r(x)} \abs{\nabla_A\Phi}^2 + 2\epsilon^{-2}\abs{\mu(\Phi)}^2 + \inner{\fK \Phi}{\Phi}.
  \end{align*}
  By Hardy's inequality,
  \begin{equation*}
    \int_{B_r(x)} \abs{\Phi}^2
    \leq
    cr^2\int_{B_r(x)} \abs{\nabla_A\Phi}^2
    + cr \int_{\del B_r(x)} \abs{\Phi}^2.
  \end{equation*}
  Therefore and because $H_{x,r} \geq \frac2r - cr$,
  for $r \in [0,r_0]$,
  \begin{equation*}
    \frac{\rd}{\rd r} \int_{\del B_r(x)} \abs{\Phi}^2 \geq 0.
  \end{equation*}
  This implies the assertion.
\end{proof}

\subsection{\texorpdfstring{$L^\infty$}{Linfty} bounds on \texorpdfstring{$\Phi$}{Phi}}

To state the next result,
we define the following variant of the Morrey norm
\begin{equation*}
  \Abs{f}_{L_\star^{p,\lambda}(U)}
  \coloneq
  \sup_{y \in U} \Abs{r_y^{-\lambda/p}f}_{L^p(U)}
\end{equation*}
with $r_y \coloneq d(y,\cdot)$.

\begin{prop}
  \label{Prop_LInftyBounds}
  If $A \in \sA(\fs,B)$, $\Phi \in \Gamma(\bS)$, and $\epsilon > 0$ satisfy \autoref{Eq_UnnormalizedBlownupSeibergWitten},
  then
  \begin{equation*}
    \Abs{\Phi}_{L^\infty(M)}
    + \Abs{\nabla_A\Phi}_{L_\star^{2,1}(M)}
    + \epsilon^{-1}\Abs{\mu(\Phi)}_{L_\star^{2,1}(M)}
    \leq c\Abs{\Phi}_{L^2};
  \end{equation*}
  moreover, for every $x \in M$, $r \in (0,r_0]$,
  \begin{align*}
    \Abs{\Phi}_{L^\infty\paren*{B_{r/2}(x)}}
    + \Abs{\nabla_A\Phi}_{L_\star^{2,1}\paren*{B_{r/2}(x)}}
    + \epsilon^{-1}\Abs{\mu(\Phi)}_{L_\star^{2,1}\paren*{B_{r/2}(x)}}
    \leq
      cm_x^\Phi(r)^{1/2}.
  \end{align*}
\end{prop}

\begin{proof}
  Let $\chi \in C_0^\infty(B_r(x),[0,1])$ be a cut-off function satisfying $\chi|_{B_{r/2}(x)} = 1$ and
  \begin{equation*}
    r\abs{\nabla\chi} \leq c \qandq
    r^2\abs{\nabla^2\chi} \leq c.
  \end{equation*}
  Denote by $G$ the Green's kernel for $B_r(x)$ and,
  for $y \in B_r(x)$,
  set $G_y \coloneq G(y,\cdot)$.
  Multiplying \autoref{Eq_LaplacianOfAbsPhiSquaredEqualsLowerOrderTerm} with $\chi^2 G_y$ and integrating by parts yields
  \begin{equation*}
    \frac12\chi(y)^2\abs{\Phi}^2(y)
    + \int_{B_r(x)} \chi^2 G_y\(\abs{\nabla_A\Phi}^2 + 2\epsilon^{-2}\abs{\mu(\Phi)}^2\) \\
    =
      \int_{B_r(x)} \chi^2 G_y \inner{\fR\Phi}{\Phi} + \Theta_y \cdot \abs{\Phi}^2
  \end{equation*}
  with
  \begin{equation*}
    \Theta_y \coloneq \inner{\nabla\chi^2}{\nabla G_y} - \frac12\Delta \chi^2\cdot G_y.
  \end{equation*}
  From
  \begin{equation*}
    \int_{B_r(x)} G_y \leq c r^2 \qandq
    \Abs{\Theta_y}_{L^\infty} \leq cr^{-3}
  \end{equation*}
  it follows that
  \begin{equation*}    
    \Abs{\chi\Phi}_{L^\infty}^2
     + \sup_{y \in B_{r/2}(x)}\int_{B_{r/2}(x)} r_y^{-1}\paren*{\abs{\nabla_A\Phi}^2 + \epsilon^{-2}\abs{\mu(\Phi)}^2}
    \leq
      cr^2 \Abs{\chi \Phi}_{L^\infty}^2
      + cr^{-3}\Abs{\Phi}_{L^2(B_r(x))}^2.
  \end{equation*}
  After rearranging and
  by \autoref{Prop_L2Bounds},
  the asserted inequalities follow.
\end{proof}

\subsection{\texorpdfstring{$W^{2,2}$}{W22} bounds on \texorpdfstring{$\Phi$}{Phi}}

\begin{prop}
  \label{Prop_W22Bounds}
  For every $c_F,c_\Phi,c_\scN > 0$ and $\delta \in (0,\frac12]$,
  there is a constant $c = c(c_F,c_\Phi,c_\scN) > 0$ such that the following holds for every $x \in M$, $r \in (0,r_0]$.
  If $A \in \sA(\fs,B)$, $\Phi \in \Gamma(\bS)$, and $\epsilon > 0$ satisfy \autoref{Eq_UnnormalizedBlownupSeibergWitten},
  \begin{equation*}
    r\int_{B_r(x)} \abs{F_{\Ad(A)}}^2 \leq c_F, \quad
    m_x^\Phi(r) \leq c_\Phi, \qandq
    \scN_x^{A,\Phi,\epsilon}(r) \leq c_\scN,
  \end{equation*}
  then
  \begin{equation*}
    r\int_{B_{r/2}(x)} \abs{\nabla_A^2\Phi}^2 
     + \paren*{\frac{\epsilon}{r}}^{2} \cdot r^3\int_{B_{r/2}(x)} \abs{\nabla_{\Ad(A)}F_{\Ad(A)}}^2
    \leq
      c.
  \end{equation*}
\end{prop}

The proof relies on the following consequence of the Lichnerowicz--Weitzenböck formula \autoref{Eq_LichnerowiczWeitzenbock}.

\begin{prop}
  \label{Prop_DeltaAbsNablaPhi}
  If $A \in \sA(\fs,B)$, $\Phi \in \Gamma(\bS)$, and $\epsilon > 0$ satisfy \autoref{Eq_UnnormalizedBlownupSeibergWitten},
  then
  \begin{multline}
    \label{Eq_DeltaAbsNablaPhi}
    \frac12\Delta\abs{\nabla_A\Phi}^2
    + \abs{\nabla_A^2\Phi}^2
    + \epsilon^{-2}\abs{\rd_{\Ad(A)}^*\mu(\Phi)}^2
    + \epsilon^{-2}\abs{\nabla_{\Ad(A)}\mu(\Phi)}^2 \\
    =
      - \epsilon^{-2}\inner{\bracket{\mu(\nabla_A\Phi,\nabla_A\Phi)}}{\mu(\Phi)}
      + 2\epsilon^{-2} \inner{\mu(\Phi)}{\rho^*(\nabla_A\Phi\wedge\nabla_A\Phi^*)}
      + \fr_{\nabla\Phi}
  \end{multline}
  with
  \begin{equation*}
    \abs{\fr_{\nabla\Phi}}
    \leq
      c\paren*{\abs{\nabla_A\Phi}^2 + \abs{\nabla_A\Phi}\abs{\Phi}}.
  \end{equation*}
\end{prop}

The proof makes use of the following observation.

\begin{prop}
  \label{Prop_[Delta,Nabla]}
  For every $A \in \sA(\fs)$ and $\Phi \in \Gamma(S)$,
  \begin{align*}
    [\nabla_A^*\nabla_A,\nabla_A]\Phi
    &=
      \rho(\rd_{\Ad(A)}^*F_{\Ad(A)})\Phi + 2\sum_{i,j=1}^3 e^i\otimes \rho\paren{F_{\Ad(A)}(e_i,e_j)}\nabla_{A,e_j}\Phi \\
    & + (\rd_A^*\fR)\Phi
      + 2\sum_{i,j=1}^3 e^i\otimes \fR(e_i,e_j)\nabla_{A,e_j}\Phi.    
  \end{align*}
\end{prop}

\begin{proof}
  This is a consequence of the following computation
  \begin{align*}
    &-\sum_{j=1}^3 \nabla_{A,e_j}\nabla_{A,e_j}\nabla_{A,e_i}\Phi \\    
    &=
      -\sum_{j=1}^3 \nabla_{A,e_j}F_A(e_j,e_i)\Phi
      -\sum_{j=1}^3 \nabla_{A,e_j}\nabla_{A,e_i}\nabla_{A,e_j}\Phi \\
    &=
      -\sum_{j=1}^3 \nabla_{A,e_j}F_A(e_j,e_i)\Phi
      -\sum_{j=1}^3 F_A(e_j,e_i)\nabla_{A,e_j}\Phi
      -\sum_{j=1}^3 \nabla_{A,e_i}\nabla_{A,e_j}\nabla_{A,e_j}\Phi \\
    &=
      (\rd_A^*F_A)(e_i)\Phi
      -2\sum_{j=1}^3 F_A(e_j,e_i)\nabla_{A,e_j}\Phi
      -\sum_{j=1}^3 \nabla_{A,e_i}\nabla_{A,e_j}\nabla_{A,e_j}\Phi.
      \qedhere
  \end{align*}
\end{proof}

\begin{proof}[Proof of \autoref{Prop_DeltaAbsNablaPhi}]
  By \autoref{Prop_LichnerowiczWeitzenbock},
  \begin{equation*}
    \nabla_A^*\nabla_A \nabla_A\Phi
    =
      [\nabla_A^*\nabla_A,\nabla_A]\Phi
      - \epsilon^{-2}\bgamma(\mu(\Phi))\nabla_A\Phi
      - \epsilon^{-2}\bgamma(\nabla_A\mu(\Phi))\Phi
      - \fK\nabla_A\Phi
      - \gamma(\nabla \fR)\Phi.
  \end{equation*}
  By \autoref{Prop_[Delta,Nabla]},
  the first term on the right-hand side can be written as
  \begin{align}
    \label{Eq_CommutatorOfLaplacianAndNablaOnPhi}
    [\nabla_A^*\nabla_A,\nabla_A]\Phi
    &=
      \epsilon^{-2}\rho(\rd_{\Ad(A)}^*\mu(\Phi))\Phi
      + 2\epsilon^{-2}\sum_{i,j=1}^3 e^i\otimes \rho\paren{\mu(\Phi)(e_i,e_j)}\nabla_{A,e_j}\Phi \\
    &\quad
      + (\rd_A^*\fR)\Phi
      + 2\sum_{i,j=1}^3 e^i\otimes \fR(e_i,e_j)\nabla_{A,e_j}\Phi. \notag
  \end{align}
  It was proved in \cite[Proposition B.4]{Doan2017a} that if $\slD_A\Phi=0$,
  then
  \begin{equation}
    \label{Eq_DAMuPhi}
    \rd_{\Ad(A)}^*\mu(\Phi) = - \rho^*(\nabla_A\Phi\Phi^*).
  \end{equation}
  These identities imply the asserted formula upon taking the inner product of \autoref{Eq_CommutatorOfLaplacianAndNablaOnPhi} with $\nabla_A\Phi$ because
  \begin{equation*}
    \inner{\bgamma(\nabla_A\mu(\Phi))\Phi}{\nabla_A\Phi}
    =
    \abs{\nabla_A\mu(\Phi)}^2
  \end{equation*}
  and
  \begin{align*}
    \sum_{i=1}^3 \inner{(\rd_{\Ad(A)}^*\mu(\Phi))(e_i)\Phi}{\nabla_{A,e_i}\Phi}
    &=
      \sum_{i=1}^3 -\inner{\rho\rho^*\paren*{\nabla_{A,e_i}\Phi\Phi^*}\Phi}{\nabla_{A,e_i}\Phi} \\
    &=
      \sum_{i=1}^3 -\inner{\rho^*\paren*{\nabla_{A,e_i}\Phi\Phi^*}}{\rho^*(\nabla_{A,e_i}\Phi\Phi^*)} \\
    &=
      -\abs{\rd_{\Ad(A)}^*\mu(\Phi)}^2.
      \qedhere
  \end{align*}
\end{proof}

\begin{proof}[Proof of \autoref{Prop_W22Bounds}]
  Let $\chi \in C_0^\infty(B_r(x),[0,1])$ be a cut-off function satisfying $\chi|_{B_{r/2}(x)} = 1$ and
  \begin{equation*}
    r\abs{\nabla\chi} \leq c \qandq
    r^2\abs{\nabla^2\chi} \leq c.
  \end{equation*}
  Multiplying \autoref{Eq_DeltaAbsNablaPhi} by $r\chi^2$, integrating by parts, and using $F_{\Ad(A)} = \epsilon^{-2}\mu(\Phi)$, yields
  \begin{align*}
    r\int_{B_r(x)} \chi^2\paren*{\abs{\nabla_A^2\Phi}^2 + \epsilon^2\abs{\nabla_AF_A}^2}
    &\leq
      c(c_\Phi,c_\scN)
      + cr\int_{B_r(x)} \abs{F_A}\cdot\chi^2\abs{\nabla_A\Phi}^2 \\
    &\leq
      c(c_\Phi,c_\scN)
      + \underbrace{c(c_F)\paren*{r\int_{B_r(x)}\chi^4\abs{\nabla_A\Phi}^4}^{1/2}}_{\eqcolon (\star)}.
  \end{align*}
  By the Gagliardo--Nirenberg interpolation inequality and the Cauchy--Schwarz inequality,
  for every $f \in C_0^\infty(B_r(x))$ and $\sigma > 0$,
  \begin{align*}
    \Abs{f}_{L^4}^2
    &\leq
      c \Abs{\nabla f}_{L^2}^{3/2}\Abs{f}_{L^2}^{1/2} \\
    &\leq
       \sigma\Abs{\nabla f}_{L^2}^2 + c(\sigma)\Abs{f}_{L^2}^2.
  \end{align*}
  Therefore,
  by Kato's inequality,
  \begin{equation*}
    (\star)
    \leq
    \frac{r}{2} \int_{B_r(x)} \chi^2\abs{\nabla_A^2\Phi}^2
    + c(c_F,c_\Phi,c_\scN).
  \end{equation*}
  Rearranging proves the asserted inequality.
\end{proof}

\subsection{Oscillation bounds on \texorpdfstring{$\Phi$}{Phi}}

\begin{prop}
  \label{Prop_OscillationBounds}
  For every $c_F,c_\Phi,c_\scN > 0$,
  there is a constant $c = c(c_F,c_\Phi,c_\scN) > 0$ such that the following holds for every $x \in M$ and $r \in (0,r_0]$.
  If $A \in \sA(\fs,B)$, $\Phi \in \Gamma(\bS)$, and $\epsilon > 0$ satisfy \autoref{Eq_UnnormalizedBlownupSeibergWitten},
  \begin{equation*}
    r\int_{B_r(x)} \abs{F_A}^2 \leq c_F, \quad
    m_x^\Phi(r) \leq c_\Phi, \qandq
    \scN_x^{A,\Phi,\epsilon}(r) \leq c_\scN,
  \end{equation*}
  then,
  for every $y,z \in B_{r/2}(x)$,
  \begin{equation*}
    \abs*{\abs{\Phi}(y)-\abs{\Phi}(z)}
    \leq
      c\paren*{\scN_x^{A,\Phi,\epsilon}(r)m_x^\Phi(r)}^{1/8}.
  \end{equation*}
\end{prop}

\begin{proof}
  By \autoref{Prop_LInftyBounds} and \autoref{Prop_W22Bounds},
  \begin{equation*}
    \Abs{\Phi}_{L^\infty\(B_{r/2}(x)\)}^2
    \leq
      c(c_\Phi)
    \qandq
    r^{1/2}\Abs{\nabla_A^2 \Phi}_{L^2\(B_{r/2}(x)\)}
    \leq
      c(c_F,c_\Phi,c_\scN).
  \end{equation*}
  Therefore,
  by Morrey's inequality and the Gagliardo--Nirenberg interpolation inequality,
  \begin{align*}
    r^{1/4}[\abs{\Phi}]_{C^{0,1/4}\(B_{r/2}(x)\)}
    &\leq
      cr^{1/4}\Abs{\abs{\nabla_A\Phi}}_{L^4\(B_{r/2}(x)\)} \\
    &\leq
      c\paren*{r^{1/2}\Abs{\abs{\nabla_A^2\Phi}}_{L^2\(B_{r/2}(x)\)}}^{3/4}
      \paren*{r^{-1/2}\Abs{\abs{\nabla_A\Phi}}_{L^2\(B_{r/2}(x)\)}}^{1/4} \\
    &\quad
      + cr^{-1/2}\Abs{\abs{\nabla_A\Phi}}_{L^2\(B_{r/2}(x)\)} \\
    &\leq
      c(c_F,c_\scN,c_\Phi)
      \paren*{\scN_x^{A,\Phi,\epsilon}(r)m_x^\Phi(r)}^{1/8}.
  \end{align*}
  This implies the assertion.
\end{proof}

\subsection{\texorpdfstring{$L^\infty$}{Linfty} bounds on \texorpdfstring{$\mu(\Phi)$}{mu(Phi)}}

\begin{prop}
  \label{Prop_LInftyBoundMu}
  For every $c_F,c_\Phi,c_\scN > 0$,
  there is a constant $c = c(c_F,c_\Phi,c_\scN) > 0$ such that the following holds for every $x \in M$, $r \in (0,r_0]$.
  If $A \in \sA(\fs,B)$, $\Phi \in \Gamma(\bS)$, and $\epsilon > 0$ satisfy \autoref{Eq_UnnormalizedBlownupSeibergWitten},
  \begin{equation*}
    r\int_{B_r(x)} \abs{F_{\Ad(A)}}^2 \leq c_F, \quad
    m_x^\Phi(r) \leq c_\Phi, \qandq
    \scN_x^{A,\Phi,\epsilon}(r) \leq c_\scN,
  \end{equation*}
  then
  \begin{equation*}
    \Abs{\mu(\Phi)}_{L^\infty\(B_{r/2}(x)\)}
    \leq
      c
      \paren*{
        \paren*{\frac{\epsilon}{r}}^2
        \scN_x^{A,\Phi,\epsilon}(r)m_x^\Phi(r)
      }^{1/32}.
  \end{equation*}
\end{prop}

\begin{proof}
  By Morrey's inequality,
  the Gagliardo--Nirenberg interpolation inequality,
  \autoref{Prop_LInftyBounds}, and
  \autoref{Prop_W22Bounds},
  \begin{align*}
    r^{1/4}[\abs{\mu(\Phi)}^2]_{C^{0,1/4}\(B_{r/2}(x)\)}
    &\leq
      cr^{1/4}\Abs{\nabla \abs{\mu(\Phi)}^2}_{L^4\(B_{r/2}(x)\)} \\
    &\leq
      c\paren*{r^{1/2} \Abs{\nabla^2 \abs{\mu(\Phi)}^2}_{L^2\(B_{r/2}(x)\)}}^{7/8}
      \paren*{r^{-3/2}\Abs{\abs{\mu(\Phi)}^2}_{L^2\(B_{r/2}(x)\)}}^{1/8} \\
    &\quad
      + cr^{-3}\Abs{\abs{\mu(\Phi)}^2}_{L^1\(B_{r/2}(x)\)} \\
    &\leq
      c(c_F,c_\Phi,c_\scN)
      \paren*{r^{-3/2}\Abs{\mu(\Phi)}_{L^2\(B_{r/2}(x)\)}}^{1/8} \\
    &\leq
      c(c_F,c_\Phi,c_\scN)
      \paren*{
        \paren*{\frac{\epsilon}{r}}^2
        \scN_x^{A,\Phi,\epsilon}(r)m_x^\Phi(r)
      }^{1/16}.
  \end{align*}
  Therefore,
  for all $y,z \in B_{r/2}(x)$,
  \begin{equation*}
    \abs*{\abs{\mu(\Phi)}^2(y) - \abs{\mu(\Phi)}^2(z)}
    \leq
    c(c_F,c_\Phi,c_\scN)
    \paren*{
      \paren*{\frac{\epsilon}{r}}^2
      \scN_x^{A,\Phi,\epsilon}(r)m_x^\Phi(r)
    }^{1/16}.
  \end{equation*}  
  This implies
  \begin{equation*}
    \abs{\mu(\Phi)}^2(y)- \fint_{B_{r/2}(x)} \abs{\mu(\Phi)}^2
    \leq
      c(c_F,c_\Phi,c_\scN)
      \paren*{
        \paren*{\frac{\epsilon}{r}}^2
        \scN_x^{A,\Phi,\epsilon}(r)m_x^\Phi(r)
      }^{1/16}.
  \end{equation*}
  Since
  \begin{equation*}
    \fint_{B_{r/2}(x)} \abs{\mu(\Phi)}^2
    \leq 
    c\paren*{\frac{\epsilon}{r}}^2 D_x^{A,\Phi,\epsilon}(r)
   =
      c\paren*{\frac{\epsilon}{r}}^2\scN_x^{A,\Phi,\epsilon}(r)m_x^\Phi(r),
  \end{equation*}
  the assertion follows.
\end{proof}

\subsection{Curvature decay}

\begin{prop}
  \label{Prop_CurvatureDecay}
  Suppose that \autoref{Hyp_GammaFControlsF} holds with $\Lambda \geq 0$.
  For every $c_F>0$,
  there are constants $r_{-1} > 0$ and $\delta_\scN = \delta_\scN(c_F) > 0$ such that the following holds for every $x \in M$ and $r \in (0,r_{-1}]$.
  If $A \in \sA(\fs,B)$, $\Phi \in \Gamma(\bS)$, and $\epsilon > 0$ satisfy \autoref{Eq_UnnormalizedBlownupSeibergWitten},
  \begin{equation*}
    r\int_{B_r(x)} \abs{F_{\Ad(A)}}^2 \leq c_F, \qandq
    \scN_x^{A,\Phi,\epsilon}(r) \leq \delta_\scN,
  \end{equation*}
  then
  \begin{equation*}
    \frac{r}{4}\int_{B_{r/4}(x)} \abs{F_{\Ad(A)}}^2
    \leq
      \Lambda + 1.
  \end{equation*}
\end{prop}

The proof relies on the following proposition regarding the decay of part of the curvature.

\begin{prop}
  \label{Prop_CurvatureComponentDecay}
  For every $c_F,c_\scN > 0$,
  there is a constant $c = c(c_F,c_\scN) > 0$ such that the following holds for every $x \in M$ and $r \in (0,r_0]$.  
  If $A \in \sA(\fs,B)$, $\Phi \in \Gamma(\bS)$, and $\epsilon > 0$ satisfy \autoref{Eq_UnnormalizedBlownupSeibergWitten},
  \begin{equation*}
    r\int_{B_r(x)} \abs{F_{\Ad(A)}}^2 \leq c_F, \quad
    m_x^\Phi(r) = 1, \qandq
    \scN_x^{A,\Phi,\epsilon}(r) \leq c_\scN,
  \end{equation*}
  then
  \begin{equation*}
    r^{1/2}\Abs{\Gamma_\Phi F_{\Ad(A)}}_{L^2(B_{r/2}(x))}
    \leq
      c
      \paren*{
        \frac{\epsilon}{r}
        + \scN_x^{A,\Phi,\epsilon}(r)^{1/8}
        + r^2
      }.
  \end{equation*}  
\end{prop}

\begin{proof}[Proof of \autoref{Prop_CurvatureDecay}]
  If $A,\Phi,\epsilon$ satisfy \autoref{Eq_UnnormalizedBlownupSeibergWitten},
  then so do
  \begin{equation*}
    A, \quad
    m_x^{\Phi}(r)^{-1/2}\cdot\Phi, \quad
    m_x^{\Phi}(r)\cdot\epsilon.
  \end{equation*}
  Moreover, $\scN_x^{A,\Phi,\epsilon}$ is invariant under this rescaling.
  Therefore, we can assume that
  \begin{equation*}
    m_x^\Phi(r) = 1.
  \end{equation*}
  For $\delta_\scN \ll 1$,
  it follows from \autoref{Prop_OscillationBounds} and \autoref{Prop_LInftyBoundMu} that, on $B_{r/2}(x)$,
  \begin{equation*}
    \frac12 \leq  \abs{\Phi} \leq 2 \qandq
    \abs{\mu(\Phi)} \leq \delta_\mu
  \end{equation*}
  with $\delta_\mu$ as in \autoref{Hyp_GammaFControlsF}.
    
  If $\epsilon/r \ll 1$,
  then the desired estimate follows from \autoref{Prop_CurvatureComponentDecay} and \autoref{Hyp_GammaFControlsF};
  otherwise,
  it follows from
  \begin{equation*}
    r\int_{B_r(x)} \abs{F_{\Ad(A)}}^2
    \leq
      \paren*{\frac{r}{\epsilon}}^2D_x^{A,\Phi,\epsilon}(r)
    =
      \paren*{\frac{r}{\epsilon}}^2\scN_x^{A,\Phi,\epsilon}(r).
    \qedhere
  \end{equation*}
\end{proof}

The proof of \autoref{Prop_CurvatureComponentDecay} relies on the following proposition,
which is a consequence of the Lichnerowicz--Weitzenböck formula \autoref{Eq_LichnerowiczWeitzenbock}.

\begin{prop}
  \label{Prop_DeltaAbsMuPhi}
  If $A \in \sA(\fs)$, $\Phi \in \Gamma(\bS)$, and $\epsilon > 0$ satisfy \autoref{Eq_UnnormalizedBlownupSeibergWitten},
  then
  \begin{equation}
    \label{Eq_DeltaAbsMuPhi}
    \frac12\Delta \abs{\mu(\Phi)}^2 + \epsilon^{-2}\abs{\Gamma_\Phi\mu(\Phi)}^2 + \abs{\nabla_A\mu(\Phi)}^2
    =
    - 2\inner{\bracket{\mu(\nabla_A\Phi,\nabla_A\Phi)}}{\mu(\Phi)}
    - \inner{\Gamma_\Phi\mu(\Phi)}{\fK\Phi}.
  \end{equation}
\end{prop}

The proof makes use of the following identity regarding the symmetric bilinear form associated with the quadratic map $\mu$.

\begin{prop}
  \label{Prop_MuMuPhiPhiPhi}
  For every $\Phi \in \Gamma(\bS)$,
  \begin{equation*}
    \mu(\bgamma(\mu(\Phi))\Phi,\Phi)
    =
    \tfrac12 \Gamma_\Phi^*\Gamma_\Phi\mu(\Phi).
  \end{equation*}
\end{prop}

\begin{proof}
  For every $\zeta \in \Omega^2(M,\Ad(\fs))$,
  \begin{align*}
    \inner{\mu(\bgamma(\mu(\Phi))\Phi,\Phi)}{\zeta}
    &=
      \frac12\inner{\bgamma^*(\bgamma(\mu(\Phi))\Phi\Phi^*)}{\zeta} \\
    &=
      \frac12\inner{\bgamma(\mu(\Phi))\Phi)}{\bgamma(\zeta)\Phi} \\
    &=
      \frac12\inner{\Gamma_\Phi\mu(\Phi)}{\Gamma_\Phi\zeta} \\
    &=
      \frac12\inner{\Gamma_\Phi^*\Gamma_\Phi\mu(\Phi)}{\zeta}.
      \qedhere
  \end{align*}
\end{proof}

\begin{proof}[Proof of \autoref{Prop_DeltaAbsMuPhi}]
  By \autoref{Prop_LichnerowiczWeitzenbock} and \autoref{Prop_MuMuPhiPhiPhi},
  \begin{align*}
    \nabla_{\Ad(A)}^*\nabla_{\Ad(A)} \mu(\Phi)
    &=
      2\mu(\nabla_A^*\nabla_A\Phi,\Phi)
      - 2\bracket{\mu(\nabla_A\Phi,\nabla_A\Phi)} \\
    &=
      - 2\epsilon^{-2}\mu(\bgamma(\mu(\Phi))\Phi,\Phi)
      - 2\mu(\fK\Phi,\Phi)
      - 2\bracket{\mu(\nabla_A\Phi,\nabla_A\Phi)} \\
    &=
      - \epsilon^{-2}\Gamma_\Phi^*\Gamma_\Phi\mu(\Phi)
      - 2\mu(\fK\Phi,\Phi)
      - 2\bracket{\mu(\nabla_A\Phi,\nabla_A\Phi)}.
  \end{align*}
  This implies the asserted formula upon taking the inner product with $\mu(\Phi)$ because
  \begin{align*}
    2\inner{\mu(\fK\Phi,\Phi)}{\mu(\Phi)}
    &=
      \frac12\inner{\bgamma^*(\fK\Phi\Phi^*)}{\bgamma^*(\Phi\Phi^*)} \\
    &=
      \frac12\inner{\fK\Phi}{\bgamma(\bgamma^*(\Phi\Phi^*))\Phi} \\
    &=
      \inner{\fK\Phi}{\Gamma_\Phi\mu(\Phi)}.
      \qedhere
  \end{align*}
\end{proof}

\begin{proof}[Proof of \autoref{Prop_CurvatureComponentDecay}]
  Let $\chi \in C_0^\infty(B_r(x),[0,1])$ be a cut-off function supported in $B_r(x)$ and satisfying $\chi|_{B_{r/2}(x)} = 1$ and
  \begin{equation*}
    r\abs{\nabla\chi} \leq c \qandq
    r^2\abs{\nabla^2\chi} \leq c.
  \end{equation*}
  Multiplying \autoref{Eq_DeltaAbsMuPhi} by $r\epsilon^{-2}\chi^2$, integrating by parts, and using $F_{\Ad(A)} = \epsilon^{-2}\mu(\Phi)$, yields
  \begin{align*}
    &r\int_{B_r(x)} \chi^2\abs{\Gamma_\Phi F_{\Ad(A)}}^2
    + \paren*{\frac{\epsilon}{r}}^{2} \cdot r^3\int_{B_r(x)} \chi^2\abs{\nabla_{\Ad(A)} F_{\Ad(A)}}^2 \\
    &\qquad\leq
      c\, \paren*{\frac{\epsilon}{r}}^{2} \cdot r \int_{B_r(x)} \abs{F_{\Ad(A)}}^2 
      + rc\int_{B_r(x)} \chi^2\abs{\nabla_A\Phi}^2\abs{F_{\Ad(A)}}
      + rc\int_{B_r(x)} \chi^2\abs{\Gamma_\Phi F_{\Ad(A)}}\abs{\Phi}.
  \end{align*}
  By the hypotheses and using rearrangement,
  \begin{align*}
    &r\int_{B_r(x)} \chi^2\abs{\Gamma_\Phi F_{\Ad(A)}}^2
      + r\epsilon^2\int_{B_r(x)} \chi^2\abs{\nabla_{\Ad(A)} F_{\Ad(A)}}^2 \\
    &\qquad\leq
      c(c_F) \paren*{\frac{\epsilon}{r}}^{2}
      + c(c_F) \paren*{r\int_{B_r(x)} \chi^2\abs{\nabla_A\Phi}^4}^{1/2}
      + cr^4.
  \end{align*}
  As in the proof of \autoref{Prop_OscillationBounds},
  the second term on the right-hand side can be bounded by
  \begin{equation*}
    c(c_F,c_\scN)\cdot \scN_x^{A,\Phi,\epsilon}(r)^{1/4}.
  \end{equation*}
  Therefore,
  \begin{equation*}
    r^{1/2}\Abs{\Gamma_\Phi F_{\Ad(A)}}_{L^2(B_{r/2}(x))}
    \leq
    c(c_F,c_\scN)
    \paren*{
      \frac{\epsilon}{r}
      + \scN_x^{A,\Phi,\epsilon}(r)^{1/8}
      + r^2
    }.
    \qedhere
  \end{equation*}
\end{proof}


%% file: RegularityScale.tex
\section{The regularity scale}
\label{Sec_RegularityScale}

Throughout this section,
suppose that \autoref{Hyp_GammaFControlsF} holds with $\Lambda \geq 0$.

\begin{definition}
  \label{Def_RegularityScale}
  For $\delta > 0$ as in the upcoming \autoref{Lem_DecayImpliesInteriorBound},
  set
  \begin{equation*}
    c_F \coloneq \delta^{-1}(\Lambda+1).
  \end{equation*}
  The \defined{regularity scale} of $A \in \sA(\fs,B)$ is the function $r_A \co M \to [0,r_0]$ defined by
  \begin{equation*}
    r_A(x)
    \coloneq
    \sup \set*{ r \in [0,r_0] : r\int_{B_r(x)} \abs{F_A}^2 \leq c_F }.
  \end{equation*}
\end{definition}

The following result is the key to the proof of \autoref{Thm_AbstractCompactness}.

\begin{prop}
  \label{Prop_RegularityScaleBound}
  There are constants $\delta,r_{-1},c>0$ such that the following holds.
  If $A \in \sA(\fs,B)$, $\Phi \in \Gamma(\bS)$, and $\epsilon > 0$ satisfy \autoref{Eq_UnnormalizedBlownupSeibergWitten} and \autoref{Eq_NormalizationOfTheSpinor},
  then
  \begin{equation*}
    r_A(x) \geq \min\set*{c^{-1}\abs{\Phi}(x)^{1/\delta},r_{-1}}.
  \end{equation*}
\end{prop}

The four upcoming subsections
analyze the frequency function.
Throughout,
let $x \in M$ and let  $A \in \sA(\fs,B)$, $\Phi \in \Gamma(\bS)$, and $\epsilon > 0$ be a solution of \autoref{Eq_UnnormalizedBlownupSeibergWitten}.
To simplify notation,
we drop the super-scripts in \autoref{Def_Frequency} and simply write $r_{-1,x}$, $m_x$, $D_x$, and $\scN_x$.

\subsection{Almost monotonicty of \texorpdfstring{$\scN$}{N}}

The following is the key result regarding the frequency function.

\begin{prop}
  \label{Prop_N'}
  For every $r \in (r_{-1,x},r_0]$,
  \begin{equation}
    \label{Eq_N'}
    \begin{split}
      \scN_x'(r)
      &\geq
        \frac{1}{2\pi r m_x(r)}
        \int_{\del B_r(x)}
          \abs{\nabla_{A,\del_r}\Phi - \frac1r \scN_x(r)\Phi}^2
          + \epsilon^{-2}\abs{i(\del_r)\mu(\Phi)}^2 \\
      &\quad
        -cr(1+\scN_x(r)).
    \end{split}
  \end{equation}
\end{prop}

Before embarking on the proof of \autoref{Prop_N'},
let us record the following consequence.

\begin{prop}
  \label{Prop_NAlmostIncreasing}
  For every $r_{-1,x} < s \leq r \leq r_0$,
  \begin{equation*}    
    \scN_x(s)
    \leq
    \paren*{1+cr^2}\scN_x(r) + cr^2.
  \end{equation*}
\end{prop}

\begin{proof}
  By \autoref{Prop_N'},
  \begin{equation*}
    \frac{\rd}{\rd r} e^{\frac12 cr^2}\paren*{\scN_x(r)+1} \geq 0.
  \end{equation*}
  This implies
  \begin{equation*}
    \scN_x(s) \leq e^{\frac12 c(r^2-s^2)}\scN_x(r) + e^{\frac12 c(r^2-s^2)}-1.
    \qedhere
  \end{equation*}
\end{proof}

The proof of of \autoref{Prop_N'} relies on the following three propositions.

\begin{prop}
  \label{Prop_D'}
  For every $r \in (0,r_0]$,
  \begin{equation*}
    D_x'(r)
    =
    \frac{1}{2\pi r}\int_{\del B_r(x)} \abs{\nabla_{A,\del_r}\Phi}^2 + \epsilon^{-2}\abs{i_{\del_r}\mu(\Phi)}^2
    + \fr_{D'}
  \end{equation*}
  with
  \begin{equation*}
    \abs{\fr_{D'}} \leq cr\paren*{D_x(r) + m_x(r)}.
  \end{equation*}
\end{prop}

\begin{proof}
  Following \citet[Proof of Lemma 5.2]{Taubes2012},
  define the tensor field $T \in \Gamma(S^2\Tstar M)$ by
  \begin{equation*}
    T = T_\Phi + \epsilon^{-2}T_\mu
  \end{equation*}
  with
  \begin{align*}
    T_\Phi(v,w)
    &\coloneq
      \inner{\nabla_{A,v}\Phi}{\nabla_{A,w}\Phi}
      - \frac12\inner{v}{w}\abs{\nabla_A\Phi}^2
      \qand \\
    T_\mu(v,w)
    &\coloneq
      \inner{i_v\mu(\Phi)}{i_w\mu(\Phi)}
      - \inner{v}{w}\abs{\mu(\Phi)}^2.
  \end{align*}  
  By a straight-forward computation,
  \begin{equation}
    \label{Eq_TrT}
    -2\tr T
    =
    \abs{\nabla_A\Phi}^2 + 2\epsilon^{-2}\abs{\mu(\Phi)}^2.
  \end{equation}
  A further computation, which we postpone for a moment, shows that
  \begin{equation}
    \label{Eq_DivT}
    \abs{\nabla^*T} \leq c\abs{\fR}\abs{\Phi}\abs{\nabla_A\Phi}.
  \end{equation}
  
  By \autoref{Eq_TrT},
  the identity
  \begin{equation*}
    \int_{B_r(x)} \inner{\nabla^*T}{\rd r_x^2}
    =
    \int_{B_r(x)} \inner{T}{\Hess(r_x^2)}
    - 2r\int_{\del B_r(x)} T(\del_r,\del_r)
  \end{equation*}
  can be rewritten as
  \begin{align*}
    \int_{B_r(x)} 2r_x \nabla^*T(\del_r) 
    &=
      - \int_{B_r(x)} \abs{\nabla_A\Phi}^2 + 2\epsilon^{-2}\abs{\mu(\Phi)}^2
      + \int_{B_r(x)} \inner{T}{\fr_\rII} \\
    &\quad
      -2r\int_{\del B_r(x)} \abs{\nabla_{A,\del_r}\Phi}^2 + \epsilon^{-2}\abs{i_{\del_r}\mu(\Phi)}^2
      +r\int_{\del B_r(x)} \abs{\nabla_A\Phi}^2 + 2\epsilon^{-2}\abs{\mu(\Phi)}^2
  \end{align*}
  with
  \begin{equation*}
    \fr_\rII \coloneq \Hess(r_x^2) - 2g.    
  \end{equation*}
  Since
  \begin{equation*}
    D_x'(r)
    =
    - \frac{1}{4\pi r^2}
    \int_{B_r(x)} \abs{\nabla_A\Phi}^2 + 2\epsilon^{-2}\abs{\mu(\Phi)}^2
    + \frac{1}{4\pi r}
    \int_{\del B_r(x)} \abs{\nabla_A\Phi}^2 + 2\epsilon^{-2}\abs{\mu(\Phi)}^2,
  \end{equation*}
  the above can be rewritten as
  \begin{align*}
    D_x'(r)
    =
    \frac{1}{2\pi r}\int_{\del B_r(x)} \abs{\nabla_{A,\del_r}\Phi}^2 + \epsilon^{-2}\abs{i_{\del_r}\mu(\Phi)}^2
    + \frac{1}{4\pi r^2} \int_{B_r(x)} 2r_x \nabla^*T(\del_r) - \inner{T}{\fr_\rII}.
  \end{align*}  
  The inequality \autoref{Eq_DivT} thus implies the assertion.

  It remains to prove \autoref{Eq_DivT}.
  Let $y \in M$ be an arbitrary point of $M$ and let $e_1,e_2,e_3$ be a local orthonormal frame such that $(\nabla_{e_i}e_j)(y) = 0$.
  All of the following computations take place at the point $y$.  
  By the Lichnerowicz--Weitzenböck formula \autoref{Eq_LichnerowiczWeitzenbock},
  \begin{align*}
    (\nabla^* T_\Phi)(e_i)
    &=
      -\sum_{j=1}^3
        \inner{\nabla_{A,e_j}\nabla_{A,e_j}\Phi}{\nabla_{A,e_i}\Phi}
      + \inner{\nabla_{A,e_j}\nabla_{A,e_i}\Phi}{\nabla_{A,e_j}\Phi}
      - \inner{\nabla_{A,e_i}\nabla_{A,e_j}\Phi}{\nabla_{A,e_j}\Phi} \\
    &=
      \inner{\nabla_A^*\nabla_A\Phi}{\nabla_{A,e_i}\Phi}
      +\sum_{j=1}^3 \inner{F_A(e_i,e_j)\Phi}{\nabla_{A,e_j}\Phi} \\
    &=
      -\epsilon^{-2}\inner{\bgamma(\mu(\Phi))\Phi}{\nabla_{A,e_i}\Phi}
      +\epsilon^{-2}\sum_{j=1}^3 \inner{\rho(\mu(\Phi)(e_i,e_j))\Phi}{\nabla_{A,e_j}\Phi}
      + \fr_T(e_i)
  \end{align*}
  with
  \begin{equation*}
    \fr_T(v)
    \coloneq
    -\inner{\fK\Phi}{\nabla_{A,v}\Phi}
    +\sum_{i=1}^3 \inner{\fR(v,e_i)\Phi}{\nabla_{A,e_i}\Phi}.
  \end{equation*}
  The first two terms on the right-hand side of the above identity can be rewritten as follows.
  By definition of $\mu(\Phi)$,
  \begin{align*}
    \inner{\bgamma(\mu(\Phi))\Phi}{\nabla_{A,e_i}\Phi}
    &=
      \inner{\mu(\Phi)}{\nabla_{\Ad(A),e_i}\mu(\Phi)} \\
    &=
      \frac12\nabla_{e_i}\abs{\mu(\Phi)}^2.
  \end{align*}
  Furthermore,
  the identity \autoref{Eq_DAMuPhi} implies that
  \begin{align*}
    \inner{\rho(\mu(\Phi)(e_i,e_j))\Phi}{\nabla_{A,e_j}\Phi}
    &=
      \inner{\mu(\Phi)(e_i,e_j)}{\rho^*((\nabla_{A,e_j}\Phi)\Phi^*)} \\
    &=
      -\inner{\mu(\Phi)(e_i,e_j)}{(\rd_{\Ad(A)}^*\mu(\Phi))(e_j)}.
  \end{align*}
  Therefore,
  \begin{equation}
    \label{Eq_DivTPhi}
    (\nabla^*T_\Phi)(v)
    =
    - \tfrac12\epsilon^{-2}\nabla_v\abs{\mu(\Phi)}^2
    - \epsilon^{-2}\inner{\rd_{\Ad(A)}^*\mu(\Phi)}{i(v)\mu(\Phi)}
    + \fr_T(v).
  \end{equation}
  
  The term $\fr_T$ satisfies the asserted estimate.
  Thus, it remains to show that the first two term on the right-hand side of \autoref{Eq_DivTPhi} are equal to $-\epsilon^{-2}\nabla^*T_\mu$.
  A brief computation shows that $\rd_A\mu(\Phi) = 0$ implies
  \begin{equation*}
    \sum_{j=1}^3 \inner{\nabla_{\Ad(A),e_j}i(e_i)\mu(\Phi)}{i(e_j)\mu(\Phi)}
    =
    \frac12\nabla_{e_i}\abs{\mu(\Phi)}^2.
  \end{equation*}
  Therefore,
  \begin{align*}
    (\nabla^*T_\mu)(e_i)
    &=
      \nabla_{e_i}\abs{\mu(\Phi)}^2
      -\sum_{j=1}^3 \inner{\nabla_{\Ad(A),e_j}i(e_j)\mu(\Phi)}{i(e_i)\mu(\Phi)} + \inner{\nabla_{\Ad(A),e_j}i(e_i)\mu(\Phi)}{i(e_j)\mu(\Phi)} \\
    &=
      \frac12\nabla_{e_i}\abs{\mu(\Phi)}^2
      + \inner{\rd_{\Ad(A)}^*\mu(\Phi)}{i(e_i)\mu(\Phi)}.
  \end{align*}
  This finishes the proof.
\end{proof}

\begin{prop}
  \label{Prop_D}
  For every $r \in (0,r_0]$,
  \begin{equation*}
    D_x(r)
    =
    \frac{1}{4\pi r}\int_{\del B_r(x)} \inner{\nabla_{A,\del_r}\Phi}{\Phi}
    + \fr_D
  \end{equation*}
  with
  \begin{equation*}
    \abs{\fr_D} \leq cr^2m_x(r).
  \end{equation*}
\end{prop}

\begin{proof}
  This is a consequence of \autoref{Cor_IntegrationByPartsFormula} with $f = 1$ and $U = B_r(x)$ and \autoref{Prop_L2Bounds}.
\end{proof}

\begin{prop}
  \label{Prop_M'}
  For every $r \in (0,r_0]$,
  \begin{equation*}
    m_x'(r) = \frac{2D_x(r)}{r} + \fr_{m'}
  \end{equation*}
  with
  \begin{equation*}
    \abs{\fr_{m'}} \leq crm_x(r).
  \end{equation*}
\end{prop}

\begin{proof}
  Denote by $H_{x,r}$ the mean curvature of $\del B_r(x)$.
  By \autoref{Cor_IntegrationByPartsFormula},
  \begin{align*}
    m_x(r)'
    &=
      \frac{1}{2\pi r^2} \int_{\del B_r(x)} \paren*{H_{x,r}-\tfrac{1}{r}}\abs{\Phi}^2
      + \frac{1}{4\pi r^2} \int_{\del B_r(x)} \del_r\abs{\Phi}^2 \\
    &=
      \frac{2D_x(r)}{r}
      + \frac{1}{2\pi r^2} \int_{\del B_r(x)} \paren*{H_{x,r}-\tfrac1r} \abs{\Phi}^2
      - \frac{2\fr_D}{r}.
  \end{align*}
  The assertion follows since $\abs*{H_{x,r}-\frac2r} \leq cr$.
\end{proof}

\begin{cor}
  \label{Cor_MAlmostIncreasing}
  For every $x \in M$ and $0 < s < r \leq r_0$,
  \begin{equation*}
    m_x(s)
    \leq
    \paren*{1+cr^2}m_x(r).
  \end{equation*}
\end{cor}

\begin{proof}[Proof of \autoref{Prop_N'}]
  By \autoref{Prop_D'} and \autoref{Prop_M'},
  \begin{align*}
    \scN_x'(r)
    &=
      \frac{D_x'(r)}{m_x(r)} - \frac{D_x(r)m_x'(r)}{m_x(r)^2} \\
    &=
      \frac{1}{2\pi r m_x(r)}\int_{\del B_r(x)} \abs{\nabla_{A,\del_r}\Phi}^2 + \epsilon^{-2}\abs{i_{\del_r}\mu(\Phi)}^2 - \frac{2D_x(r)^2}{rm_x(r)^2}
      + \frac{\fr_{D'}}{m_x(r)}
      - \frac{\fr_{m'}}{m_x(r)}\scN_x(r).
  \end{align*}  
  By  \autoref{Prop_D},
  \begin{align*}
    &\frac{1}{2\pi rm_x(r)}\int_{\del B_r(x)} \abs{\nabla_{A,\del_r}\Phi - \frac1r \scN_x(r)\Phi}^2 \\
    &\quad=
      \frac{1}{2\pi rm_x(r)}\int_{\del B_r(x)} \abs{\nabla_{A,\del_r}\Phi}^2
      - \frac{\scN_x(r)}{\pi r^2m_x(r)} \int_{\del B_r(x)} \inner{\nabla_{A,\del_r}\Phi}{\Phi}
      + \frac2r \scN_x(r)^2 \\
    &\quad=
      \frac{1}{2\pi rm_x(r)}\int_{\del B_r(x)} \abs{\nabla_{A,\del_r}\Phi}^2
      - \frac2r \scN_x(r)^2
      + \frac{4\fr_D}{rm_x(r)}\scN_x(r).
  \end{align*}
  Therefore,
  \begin{align*}
    \scN_x'(r)
    &=
      \frac{1}{2\pi r m_x(r)}\int_{\del B_r(x)} \abs{\nabla_{A,\del_r}\Phi}^2 + \epsilon^{-2}\abs{i_{\del_r}\mu(\Phi)}^2
      - \frac{2}{r}\scN_x(r)^2
      + \frac{\fr_{D'}}{m_x(r)}
      - \frac{\fr_{m'}}{m_x(r)}\scN_x(r) \\
    &=
      \frac{1}{2\pi rm_x(r)}\int_{\del B_r(x)} \abs{\nabla_{A,\del_r}\Phi - \frac1r \scN_x(r)\Phi}^2 + \epsilon^{-2}\abs{i_{\del_r}\mu(\Phi)}^2 \\
    &\quad
      + \underbrace{\frac{\fr_{D'}}{m_x(r)} - \paren*{\frac{4\fr_D}{rm_x(r)} + \frac{\fr_{m'}}{m_x(r)}}\scN_x(r)}_{\eqcolon\,\star}.
  \end{align*}
  This completes the proof since $\abs{\star} \leq cr\paren*{1+\scN_x(r)}$.
\end{proof}

\subsection{\texorpdfstring{$\scN$}{N} controls the growth of \texorpdfstring{$m$}{m}}

\begin{prop}
  \label{Prop_NControlsGrowthOfM}
  For every $x \in M$ and $0 < s < r \leq r_0$,
  \begin{equation*}
    \paren*{\frac{r}{s}}^{(2-cr^2)\scN_x(s)-cr^2}m_x(s)
    \leq m_x(r)
    \leq \paren*{\frac{r}{s}}^{(2+cr^2)\scN_x(r)+cr^2}m_x(s).
  \end{equation*}
\end{prop}

\begin{proof}
  By \autoref{Prop_NAlmostIncreasing} and \autoref{Prop_M'},
  for $t \in [s,r]$,
  \begin{align*}
    \frac{\rd}{\rd t}\log m_x(t)
    &\leq
      \frac{2\scN_x(t)}{t} + ct \\
    &\leq
      \frac{2(1+cr^2)}{t}\scN_x(r)
      + \frac{cr^2}{t}
  \end{align*}
  as well as
  \begin{equation*}
    \frac{\rd}{\rd t}\log m_x(t)
    \geq
    \frac{2(1-cr^2)}{t}\scN_x(s)
    - \frac{cr^2}{t}.
  \end{equation*}  
  These integrate to the asserted inequalities.
\end{proof}

\begin{prop}
  \label{Prop_M>0}
  If $\Phi \neq 0$,
  then,
  for every $x \in M$ and $r \in (0,r_0]$,
  \begin{equation*}
    m_x(r) > 0;
  \end{equation*}
  in particular, $r_{-1,x} = 0$.
\end{prop}

\begin{proof}  
  If $m_x(r) = 0$,
  for some $r \in (0,r_0]$,
  then it follows from \autoref{Prop_NControlsGrowthOfM} that $m_x = 0$.
  Therefore, $\Phi$ vanishes on $B_{r_0}(x)$.
  This in turn implies that $m_y(r_0/2)$ vanishes for all $y \in B_{r_0/2}(x)$.
  Hence, $\Phi$ vanishes on $B_{\frac32 r_0}(x)$.
  Repeating this argument shows that $\Phi$ vanishes on all of $M$.
\end{proof}

\subsection{Frequency bounds}

\begin{prop}
  \label{Prop_PhiControlsN}
  For every $r_\star \in (0,r_0]$ and $\delta > 0$,
  if
  \begin{equation*}
    0 < s \leq r_\star \min\set*{1,\paren*{\frac{\abs{\Phi}^2(x)}{2m_x(r_\star)}}^{1/\delta}},
  \end{equation*}
  then
  \begin{equation*}
    \scN_x(s) \leq 2\delta + r_\star.
  \end{equation*}
\end{prop}

\begin{proof}
  By \autoref{Cor_MAlmostIncreasing} and \autoref{Prop_NControlsGrowthOfM},
  for every $s \in (0,r_\star]$,  
  \begin{equation*}
    \paren*{\frac{r_\star}{s}}^{\scN_x(s)-r_\star}
    \abs{\Phi}^2(x)
    \leq
    2m_x(r_\star).
  \end{equation*}
  Therefore,
  \begin{equation*}
    \scN_x(s)
    \leq    \frac{\log\paren*{\frac{2m_x(r_\star)}{\abs{\Phi}^2(x)}}}{\log\paren*{\frac{r_\star}{s}}}
    + r_\star.
  \end{equation*}
  This implies the asserted inequality.
\end{proof}

\subsection{Varying the base-point}

\begin{prop}
  \label{Prop_NDependenceOnX}
  There is a constant $c > 0$ such that,
  for every $x \in M$ and $r \in (0,r_0/4]$,
  if $\scN_x(4r) \leq 1$,
  then,
  for every $y \in B_r(x)$ and $s \in (0,2r]$,
  \begin{equation*}
    \scN_y(s) \leq c\paren*{\scN_x(4r) + r^2}.
  \end{equation*}
\end{prop}

\begin{proof}
  Since $\scN_x(4r) \leq 1$,
  by \autoref{Prop_L2Bounds} and \autoref{Prop_NControlsGrowthOfM},
  \begin{align*}
    m_x(4r)
    \leq
      c m_x(r/2)
    \leq
      cr^{-3}\int_{B_{r}(x)} \abs{\Phi}^2
    \leq
      cr^{-3}\int_{B_{2r}(y)} \abs{\Phi}^2
    \leq
      cm_y(2r).
  \end{align*}
  Therefore,
  \begin{equation*}
    \scN_y(2r)
    \leq
      \frac{m_x(4r)}{m_y(2r)}
      \scN_x(4r) \\
    \leq
      c\scN_x(4r).
  \end{equation*}
  The assertion thus follows from \autoref{Prop_NAlmostIncreasing}.
\end{proof}

\subsection{Decay implies interior bound}

The following result is essentially contained in \cite[Proof of Lemma 6.2]{Taubes2012}.

\begin{lemma}
  \label{Lem_DecayImpliesInteriorBound}
  There is a constant $\delta > 0$ such that the following holds for every $x \in M$ and $r > 0$.
  If $f\co \bar B_r(x) \to [0,\infty)$ is an $L^1$ function such that, for every $y \in M$ and $s > 0$,
  \begin{equation}
    \label{Eq_DecayImpliesInteriorBound}
    B_s(y) \subset B_r(x) \qandq
    s\int_{B_s(y)} f \leq 1 \implies \frac{s}{4}\int_{B_{s/4}(y)} f \leq \delta,
  \end{equation}
  then
  \begin{equation*}
    \frac{r}{2}\int_{B_{r/2}(x)} f \leq 1.
  \end{equation*}
\end{lemma}

\begin{proof}
  The \defined{regularity scale} associated with $f$ is the function $r_f\co B_r(x) \to (0,\infty]$ defined by
  \begin{equation*}
    r_f(y)
    \coloneq
    \sup \set*{ s \geq 0 : s\int_{B_s(y)\cap B_r(x)} f \leq 1}.
  \end{equation*}
  If $r_f(x) < \frac{r}{2}$ and $\delta > 0$ is sufficiently small,
  then the following leads to a contradiction.

  Pick a maximal sequence $x_0, x_1, \ldots, x_N$ starting with $x_0 \coloneq x$ and such that,
  for every $n = 0,1,\ldots,N-1$,
  \begin{equation*}
    x_{n+1} \in B_{r_f(x_n)}(x_n) \qandq
    r_f(x_{n+1}) < \tfrac12 r_f(x_n).
  \end{equation*}
  Such a sequence must terminate, because otherwise $(x_n)_{n\in \N}$ converges to a point $x_\infty \in B_r(x)$ with $r_f(x_\infty)=0$, which is a contradiction.
  By maximality,
  $x_\star \coloneq x_N$ is such that,
  for every $y \in B_{r_f(x_\star)}(x_\star)$,
  \begin{equation}
    \label{Eq_RFXStar}
    \tfrac12 r_f(x_\star) \leq r_f(y).
  \end{equation}

  There is a constant $N_c \in \N$ depending only on $B_r(x)$ and a finite set $\set{y_1,\ldots,y_{N_c}} \subset B_{r_f(x_\star)}(x_\star)$ such that
  \begin{equation*}
    B_{r_f(x_\star)}(x_\star)
    \subset
    \bigcup_{n=1}^{N_c} B_{\frac18r_f(x_\star)}(y_n).
  \end{equation*}
  Since $r_f(x) < \frac{r}{2}$,
  by construction of $x_\star$,
  \begin{equation*}
    d(x,x_\star) + r_f(x_\star) + \tfrac12r_f(x_\star)
    < \sum_{n=0}^{N+1} \tfrac{1}{2^n}r_f(x_a) \leq 2r_f(x)
    < r;
  \end{equation*}
  that is:
  \begin{equation*}
    B_{\frac12r_f(x_\star)}(y_n) \subset B_r(x).
  \end{equation*}
  Therefore,
  by \autoref{Eq_DecayImpliesInteriorBound} and \autoref{Eq_RFXStar},
  \begin{equation*}
    \frac{r_f(x_\star)}{8} \int_{B_{\frac18r_f(x_\star)}(y_n)} f
    \leq \delta.
  \end{equation*}
  Hence,
  \begin{equation*}
    r_f(x_\star)\int_{B_{r_f(x_\star)}(x_\star)} f
    \leq
      r_f(x_\star)\sum_{n=1}^{N_c} \int_{B_{\frac18r_f(x_\star)}(y_n)} f
    \leq
      8N_c\delta.
  \end{equation*}
  If $\delta \leq \frac1{16}N_c$,
  then the integral on the left-hand side is at most $\frac12$.
  This, however, contradicts the definition of $r_f(x_\star)$ since $\bar B_{r_f(x_\star)}(x_\star) \subset B_r(x)$.  
\end{proof}

\subsection{Proof of \autoref{Prop_RegularityScaleBound} }
  
Without loss of generality assume that $\abs{\Phi}$ is not identically zero.
Choose $r_{-1}$ and $\delta_\scN$ as in \autoref{Prop_CurvatureDecay} with $c_F$ as in \autoref{Def_RegularityScale}.

If $r_\dagger \in (0,r_{-1}]$ is such that,
for every $B_s(y) \subset B_{r_\dagger}(x)$,
\begin{equation*}
  \scN_y(s) \leq \delta_\scN,
\end{equation*}
then, by \autoref{Prop_CurvatureDecay}, \autoref{Lem_DecayImpliesInteriorBound} applies to
\begin{equation*}
  f \coloneq \frac{\abs{F_A}^2}{c_F}.
\end{equation*}
Therefore,
\begin{equation*}
  \frac{r_\dagger}{4}\int_{B_{r_\dagger/4}(x)} \abs{F_A}^2 \leq c_F;
\end{equation*}
that is:
\begin{equation*}
  r_A(x) \geq \frac{r_\dagger}{4}.
\end{equation*}

Let $0 < \sigma \ll 1$.
By \autoref{Prop_L2Bounds},
\begin{equation*}
  m_x(r) \leq cr_0^{-3}\Abs{\Phi}_{L^2(M)}^2 = cr_0^{-3}.
\end{equation*}
By \autoref{Prop_PhiControlsN},
there is a constant $c > 0$ such that
\begin{equation*}
  \scN_x(4r_\dagger) \leq \sigma\delta_\scN + r_\dagger
  \qforq
  r_\dagger
  \coloneq
  c^{-1}\min\set*{
    1,\abs{\Phi}^{8/\sigma\delta_\scN}(x)
  }.
\end{equation*}
By \autoref{Prop_NDependenceOnX},
for every $B_s(y) \subset B_{r_\dagger}(x)$,
\begin{equation*}
  \scN_y(s) \leq c\paren*{\sigma\delta_\scN + r_\dagger}.
\end{equation*}
Therefore,
after possibly shrinking $r_\dagger$,
for every $B_s(y) \subset B_{r_\dagger}(x)$,
$\scN_y(s) \leq \delta_\scN$.
This finishes the proof.
\qed

\subsection{Hölder bounds}

\begin{prop}
  \label{Prop_HölderBound}
  Suppose that \autoref{Hyp_GammaFControlsF} holds.
  There are constants $\alpha \in (0,1)$ and $c > 0$ such that,
  for every  $A \in \sA(\fs,B)$, $\Phi \in \Gamma(\bS)$, and $\epsilon > 0$ satisfying \autoref{Eq_UnnormalizedBlownupSeibergWitten} and \autoref{Eq_NormalizationOfTheSpinor},
  \begin{equation*}
    [\Phi]_{C^{0,\alpha}(M)} \leq c.
  \end{equation*}
\end{prop}

\begin{proof}
  Let $\delta,r_{-1},c> 0$ be as in \autoref{Prop_RegularityScaleBound}.
  Set $\alpha \coloneq \min\set*{\frac14,\frac12\delta}$.
  Let $x,y \in M$ such that $\abs{\Phi}(x) \geq \abs{\Phi}(y)$.
 
  If $d(x,y) \leq r_A(x)^2$,
  then $d(x,y) \leq r_A(x)/2$ because $r_A \leq r_0$, and
  by Morrey's inequality, Kato's inequality, and \autoref{Prop_W22Bounds},
  \begin{equation*}
    [\abs{\Phi}]_{C^{0,1/2}\paren*{B_{r_A(x)/2}(x)}}
    \leq
      c\Abs{\nabla_A\Phi}_{L^6\paren*{B_{r_A(x)/2}(x)}}
    \leq
      c r_A(x)^{-1/2}
    \leq
    c d(x,y)^{-1/4}.
  \end{equation*}
  Hence,
  \begin{equation*}
    \abs{\Phi}(x)-\abs{\Phi}(y)
    \leq
      c d(x,y)^{1/4}
    \leq
      c d(x,y)^\alpha.
  \end{equation*}

  If $d(x,y) \geq r_A(x)^2$,
  then \autoref{Prop_RegularityScaleBound} either $d(x,y) \geq r_{-1}^2$ or $d(x,y) \geq c^{-2}\abs{\Phi}(x)^{2/\delta}$.
  In the first case, it follows from \autoref{Prop_LInftyBounds} that
  \begin{equation*}
    \abs{\Phi}(x)-\abs{\Phi}(y)
    \leq cd(x,y)^\alpha.
  \end{equation*}
  In the second case,
  \begin{equation*}
    \abs{\Phi}(x)-\abs{\Phi}(y)
    \leq 2\abs{\Phi}(x)
    \leq c^\delta d(x,y)^{\delta/2}.
    \qedhere
  \end{equation*}
\end{proof}


%% file: ProofOfAbstractCompactnessTheorem.tex
\section{Proof of \autoref{Thm_AbstractCompactness}}
\label{Sec_ProofOfAbstractCompactness}

Let $(A_n,\Phi_n,\epsilon_n)$ be a sequence of solutions of \autoref{Eq_UnnormalizedBlownupSeibergWitten} and \autoref{Eq_NormalizationOfTheSpinor} with $\epsilon_n$ tending to zero.
By \autoref{Prop_LInftyBounds} and \autoref{Prop_HölderBound},
for some $\alpha \in (0,1)$,
\begin{equation*}
  \Abs{\abs{\Phi_n}}_{C^{0,\alpha}} \leq c.
\end{equation*}
Therefore, after passing to a subsequence, for every $\beta \in (0,\alpha)$, $\abs{\Phi_n}$ converges to a limit in the $C^{0,\beta}$--topology.
Denote this limit by $\abs{\Phi}$ and set $Z \coloneq \abs{\Phi}^{-1}(0)$.
Since $\Abs{\Phi}_{L^2} = 1$, $Z$ is a proper subset of $M$.

By \autoref{Prop_RegularityScaleBound},
for every $x \in M\setminus Z$,
\begin{equation*}
  r_A(x) \coloneq \liminf_{n \in \N} r_{A_n}(x) > 0.
\end{equation*}
Therefore,
on every compact subset of $M\setminus Z$, the $L^2$--norms of $F_{A_n}$ are uniformly bounded;
and, up to gauge transformations and after passing to a subsequence,
$(A_n)$ can be assumed to converge in the weak $W^{1,2}$ topology to a limit $A$.
Moreover, by \autoref{Prop_W22Bounds},
after passing to a further subsequence,
$(\Phi_n)$ converges in the weak $W^{2,2}$ topology to a limit $\Phi$.
A patching argument as in \cite[Section 4.2.2]{Donaldson1990} yields asserted convergence statement on $M\setminus Z$.
By construction,
the limit $(A,\Phi)$ satisfies \autoref{Eq_LimitingSeibergWitten}.

It remains to prove that $Z$ is nowhere-dense.
The proof of this fact relies on the following.

\begin{prop}
  For every $x \in M$ and $r \in (0,r_0]$,
  \begin{align*}
    \lim_{n\to \infty} \int_{B_r(x)} \abs{\nabla_{A_n} \Phi_n}^2 + 2\epsilon_n^{-2}\abs{\mu(\Phi_n)}^2
    &=
     \int_{B_r(x)} \abs{\nabla_A \Phi}^2 \qand \\
    \lim_{n\to \infty} \int_{\del B_r(x)} \abs{\Phi_n}^2
    &=
      \int_{\del B_r(x)} \abs{\Phi}^2.
  \end{align*}
\end{prop}

\begin{proof}
  The second assertion is a consequence of the Hölder convergence.  
  To prove the first assertion, we proceed as follows.
  For $\epsilon \in (0,\frac12]$,
  set $Z_\epsilon \coloneq \abs{\Phi}^{-1}([0,\epsilon])$.
  Since weak $W^{2,2}$ convergence implies $W^{1,2}$ convergence,
  \begin{equation*}
    \lim_{n\to \infty} \int_{B_r(x)\setminus Z_\epsilon} \abs{\nabla_{A_n} \Phi_n}^2
    =
    \int_{B_r(x)\setminus Z_\epsilon} \abs{\nabla_A \Phi}^2.
  \end{equation*}
  Moreover,
  \begin{equation*}
    \lim_{n\to \infty} \int_{B_r(x)\setminus Z_\epsilon} 2\epsilon_n^{-2}\abs{\mu(\Phi_n)}^2
    =
      \lim_{n\to \infty} \int_{B_r(x)\setminus Z_\epsilon} 2\epsilon_n^2\abs{F_{\Ad(A)}}^2
    =
      0.
  \end{equation*}
  The discussion in the next paragraph shows that there exists a $\lambda > 0$ such that, for every $\epsilon > 0$,
  \begin{equation*}
    \int_{Z_\epsilon} \abs{\nabla_{A_n} \Phi_n}^2 + 2\epsilon_n^{-2}\abs{\mu(\Phi_n)}^2
    \leq c\epsilon^\lambda.
  \end{equation*}
  This together with the above implies the first assertion.

  Fix a cut-off function $\chi \in C_0^\infty([0,2),[0,1])$ with $\chi|_{[0,1]} = 1$.
  \autoref{Cor_IntegrationByPartsFormula} with $f=\chi(\epsilon^{-1}\abs{\Phi_n})$ and $U=M$, integrating the resulting term with $\Delta\abs{\Phi_n}^2$ by parts once and using Kato's inequality yields
  \begin{align*}
    \int_{Z_\epsilon} \abs{\nabla_{A_n}\Phi_n}^2 + 2\epsilon_n^{-2}\abs{\mu(\Phi_n)}^2
    &\leq
      c\int_{Z_{2\epsilon}} \abs{\Phi_n}^2
      + c\int_{Z_{2\epsilon}\setminus Z_\epsilon} \epsilon^{-1}\abs{\Phi_n}\abs*{\nabla_A\abs{\Phi_n}}^2 \\
    &\leq
      c\epsilon^2 + c \int_{Z_{2\epsilon}\setminus Z_{\epsilon}} \abs{\nabla_{A_n}\Phi_n}^2.
  \end{align*}
  Therefore,
  \begin{equation*}
    f(\epsilon) \coloneq \int_{Z_\epsilon} \abs{\nabla_{A_n}\Phi_n}^2 + 2\epsilon_n^{-2}\abs{\mu(\Phi_n)}^2
  \end{equation*}
  satisfies
  \begin{equation*}
    f(\epsilon) \leq \sigma (\epsilon^2 + f(2\epsilon))
    \qwithq
    \sigma \coloneq c/(1+c).
  \end{equation*}
  Without loss of generality $\sigma \geq 1/2$.
  By iterating the above inequality $k$ times and using that $f$ is bounded,
  \begin{align*}
    f(\epsilon)
    &\leq
      \sigma\epsilon^2 \sum_{i=0}^{k-1} (4\sigma)^i + \sigma^kf(2^k\epsilon) \\
    &\leq
      \sigma \epsilon^2 (4\sigma)^{k-1}\sum_{i=0}^{\infty} (4\sigma)^{-i} + \sigma^kf(2^k\epsilon) \\
    &\leq
      c\epsilon^2(4\sigma)^k + c\sigma^k.
  \end{align*}
  
  With $k \coloneq \floor{-\log \epsilon/\log 2}$ this gives
  \begin{equation*}
    f(\epsilon)
    \leq
      c\epsilon^{2-\log(4\sigma)/\log 2} + c\epsilon^{-\log\sigma/\log 2}
    \leq
      c\epsilon^\lambda
  \end{equation*}
  for some $\lambda>0$ depending on $\sigma$ only, since $\log(4\sigma)/\log 2 < 2$.
\end{proof}

If $Z$ failed to be nowhere-dense,
then we could find $x \in Z$ and $0 < r \leq r_0$ such that $B_r(x) \subset Z$.
By the above,
\autoref{Prop_NControlsGrowthOfM} applies and shows that, in fact, $B_{r_0}(x) \subset Z$.
This in turn implies that $m_y^\Phi(r_0/2)$ vanishes for all $y \in B_{r_0/2}(x)$.
Hence, $\abs{\Phi}$ vanishes on $B_{\frac32 r_0}(x)$.
Repeating this argument shows $Z$ to be all of $M$,
which contradicts $\Abs{\Phi}_{L^2} = 1$.
\qed


%% file: ProofOfADHM12SeibergWittenCompactness.tex
\section{Proof of \autoref{Thm_ADHM12SeibergWittenCompactness}}
\label{Sec_ProofOfADHM12SeibergWittenCompactness}

The proof of \autoref{Thm_ADHM12SeibergWittenCompactness} relies on \autoref{Thm_AbstractCompactness}.
Over the course of the next three subsections we establish that the hypotheses of the latter hold for the ADHM$_{1,2}$ Seiberg--Witten equation.
The geometric and algebraic observations made in the process also enter crucially in refining the conclusion of \autoref{Thm_AbstractCompactness} to obtain \autoref{Thm_ADHM12SeibergWittenCompactness}.

\subsection{The geometry of the adjoint representation of \texorpdfstring{$\SU(2)$}{SU(2)}}
\label{Sec_StableFlatSL2C}

The first step towards the proof of \autoref{Thm_ADHM12SeibergWittenCompactness} is to verify \autoref{Hyp_GammaFControlsF}---or, more precisely: the condition in \autoref{Rmk_GeometricCriterion}---for stable flat $\PSL_2(\C)$--connections.

\begin{lemma}
  \label{Lem_GammaXiMuControlsMu}
  There are constants $\delta_\mu,c > 0$ such that,
  for every $\bxi \in \su(2)\otimes \H$,
  if
  \begin{equation*}
    \abs{\mu(\bxi)} \leq \delta_\mu\abs{\bxi}^2,
  \end{equation*}
  then
  \begin{equation*}
    \abs{\bxi}\abs{\mu(\bxi)} \leq c\abs{\Gamma_\bxi\mu(\bxi)}.
  \end{equation*}
\end{lemma}

The proof of \autoref{Lem_GammaXiMuControlsMu} relies on the following observation
which is proved by a simple computation;
see, e.g., \cite[Proposition D.5]{Doan2017d}.

\begin{prop}
  \label{Prop_AbsMuXi=AbsCommutators}
  Denote by $\mu\co \fg\otimes\H \to \fg\otimes\Im\H$ the hyperkähler moment map associated with the adjoint representation $G \to \Sp(\fg\otimes\H)$.
  For every $\bxi = \xi_0\otimes 1 + \xi_1\otimes i + \xi_2\otimes j + \xi_3\otimes k \in \fg\otimes \H$,
  \begin{equation*}
    \abs{\mu(\bxi)}^2 = \frac12 \sum_{i,j=0}^3 \abs*{[\xi_i,\xi_j]}^2.
  \end{equation*}
\end{prop}

\begin{proof}[Proof of \autoref{Lem_GammaXiMuControlsMu}]
  Without loss of generality $\abs{\bxi} = 1$.
  The zero locus $\mu^{-1}(0)$ is a cone with smooth link.
  Therefore,
  if $\abs{\mu(\bxi)} \leq \delta_\mu \ll 1$,
  then $\bxi$ has a unique decomposition as
  \begin{equation*}
    \bxi = \bzeta + \hat\bxi
  \end{equation*}
  with
  \begin{equation*}
    \mu(\bzeta) = 0, \quad
    \hat\bxi \perp T_{\bzeta}\mu^{-1}(0), \qandq
    \abs{\hat\bxi} \leq c\abs{\mu(\bxi)} \leq c\delta_\mu \ll 1. 
  \end{equation*}
  
  By \autoref{Prop_AbsMuXi=AbsCommutators},
  \begin{equation*}
    \bzeta = \tau_0 \otimes v_0
  \end{equation*}
  with $\abs{\tau_0} = 1$.
  Extend $\tau_0$ to an orthonormal basis $\tau_0,\tau_1,\tau_2$ of $\su(2)$ such that
  \begin{equation}
    \label{Eq_TauCommutationRelations}
    [\tau_0,\tau_1] = 2\tau_2, \quad
    [\tau_1,\tau_2] = 2\tau_0, \qandq
    [\tau_2,\tau_0] = 2\tau_1.
  \end{equation}  
  Since
  \begin{equation*}
    T_{\bzeta}\mu^{-1}(0) = \Span{\tau_0}\otimes \H + \su(2)\otimes \Span{v_0},
  \end{equation*}
  it follows that
  \begin{equation}
    \label{Eq_DMuHatXiPerpMuHatXi}
    \rd_{\bzeta}\mu(\hat\bxi) = 2\mu(\bzeta,\hat\bxi) \in \Span{\tau_1,\tau_2}\otimes \Im\H \qandq
    \mu(\hat\bxi) \in \Span{\tau_0}\otimes \Im\H.
  \end{equation}  
  A simple computation using \autoref{Eq_TauCommutationRelations} shows that
  \begin{equation*}
    \Gamma_{\bzeta}\mu(\hat\bxi) = 0 \qandq
    \abs{\Gamma_{\bzeta}\mu(\bzeta,\hat\bxi)} = 2\abs{\bzeta}\abs{\mu(\bzeta,\hat\bxi)}.
  \end{equation*}
  Therefore,
  \begin{align*}
    \abs{\mu(\bxi)}^2
    &\leq
      4\abs{\mu(\bzeta,\hat\bxi)}^2 + \abs{\mu(\hat\bxi)}^2 \\
    &\leq
      c\abs{\Gamma_{\bzeta}\mu(\bxi)}^2
      + c\abs{\mu(\bxi)}^4 \\
    &\leq
      c\abs{\Gamma_{\bxi}\mu(\bxi)}^2
      + c\delta_\mu^2 \abs{\mu(\bxi)}^2.
  \end{align*}
  A rearrangement implies the asserted estimate.  
\end{proof}

\begin{remark}
  It is crucial that the link of $\mu^{-1}(0)$ is smooth.
  For $\su(r)$ with $r \geq 3$ this condition fails and, in fact,
  the conclusion of \autoref{Lem_GammaXiMuControlsMu} does not hold in this case.
\end{remark}

\subsection{The geometry of the ADHM\texorpdfstring{$_{1,2}$}{12} representation}

This subsection contains a number of geometric facts regarding the quaternionic representation of $\U(2)$ on $\H\otimes_\C\C^2 \oplus \su(2)\otimes \H$.
These will play a crucial role in the proof of \autoref{Hyp_GammaFControlsF} for ADHM$_{1,2}$ Seiberg--Witten monopoles in the next subsection.

\begin{prop}
  \label{Prop_AbsPiTMuPsi}
  Let $k \in \N$.
  Denote by $\mu \co \H\otimes_\C\C^k \to \fu(k)\otimes\Im \H$ the hyperähler moment map associated with the quaternionic representation $\U(k) \to \Sp(\H\otimes_\C\C^k)$.
  Let $\ft \subset \fu(k)$ be a maximal torus.
  Denote by $\pi_\ft \co \fu(k) \to \ft$ the orthogonal projection to $\ft$.
  For every $\Psi \in \H\otimes_\C \C^k$,
  \begin{equation*}
    \abs{\pi_\ft\mu(\Psi)} = \frac12\abs{\Psi}^2.
  \end{equation*}
\end{prop}

\begin{proof}
  For $k = 1$,
  $\ft = \fu(1)$.
  By \autoref{Eq_ClassicalSeibergWitten_MomentMap},
  \begin{align*}
    \abs{\mu(\Psi)}^2
    &=
      \frac12\inner{\bgamma(\mu(\Psi))\Psi}{\Psi} \\
    &=
      \frac14\abs{\Psi}^4.
  \end{align*}

  For $k \in \set{ 2,3,\ldots}$,
  without loss of generality $\ft = \fu(1)^{\oplus k}$.
  The composition $\pi_\ft\circ\mu$ is the hyperkähler moment map for the action of $\U(1)^k \subset \U(k)$ on $\H\otimes_\C\C^k$.
  Thus
  \begin{equation*}
    \pi_\ft\mu(\Psi_1,\ldots,\Psi_k) = (\mu(\Psi_1),\ldots,\mu(\Psi_k))
  \end{equation*}
  and the assertion follows from the case $k = 1$.
\end{proof}

\begin{prop}
  \label{Prop_MuPsiXiBoundsMuPsi+MuXi}
  There is a constant $c > 0$ such that,
  for every $\Psi \in \H\otimes_\C\C^2$ and $\bxi \in \H\otimes\su(2)$,
  \begin{equation*}
    \abs{\mu(\Psi)} + \abs{\mu(\bxi)} \leq c\abs{\mu(\Psi,\bxi)}.
  \end{equation*}
\end{prop}

This is an immediate consequence of the following.

\begin{prop}
  \label{Prop_MPsiMuXiNotLarge}
  There is a constant $\sigma < 1$ such that,
  for every $\Psi \in \H\otimes_\C\C^2$ and $\bxi \in \H\otimes\su(2)$,
  \begin{equation*}
    -\inner{\mu(\Psi)}{\mu(\bxi)} \leq \sigma\abs{\mu(\Psi)}\abs{\mu(\bxi)}.
  \end{equation*}  
\end{prop}

The proof relies on the following fact.

\begin{prop}[{\citet[Section 2.2]{Nakajima1999}}; see also {\cite[Proposition D.4]{Doan2017d}}]
  \label{Prop_MuPsiXi=0ImpliesPsi=0}
  Let $k \in \N$.
  For every $\Psi \in \H\otimes_\C\C^k$ and $\bxi \in \H\otimes\su(k)$,
  if $\mu(\Psi,\bxi) = 0$, then $\Psi = 0$.  
\end{prop}

\begin{proof}[Proof \autoref{Prop_MPsiMuXiNotLarge}]
  Set
  \begin{equation*}
    \sigma
    \coloneq
    \sup\set*{
      \inner{\mu(\Psi)}{\mu(\bxi)}
      :
      \abs{\mu(\Psi)} = \abs{\mu(\bxi)} = 1
    }.
  \end{equation*}
  By Cauchy--Schwarz,
  $\sigma \leq 1$;
  moreover, if $\sigma = 1$,
  then there are
  \begin{equation*}
    \mu_\Psi \in \overline{\set{ \mu(\Psi) : \abs{\mu(\Psi)} = 1}} \qandq
    \mu_\bxi \in \overline{\set{ \mu(\bxi) : \abs{\mu(\bxi)} = 1}}    
  \end{equation*}
  with
  \begin{equation*}
    \inner{\mu_\Psi}{\mu_\bxi} = -\abs{\mu_\Psi}\abs{\mu_\bxi}.
  \end{equation*}
  Therefore,
  \begin{equation*}
    \mu_\Psi = -\mu_\bxi.
  \end{equation*}
  The upcoming discussion proves this to be impossible.

  By \autoref{Prop_AbsPiTMuPsi},
  the set $\set{ \mu(\Psi) : \abs{\mu(\Psi)} = 1}$ is closed.
  In particular,
  \begin{equation*}
    \mu_\Psi = \mu(\Psi)
  \end{equation*}
  for some $\Psi \in \H\otimes_\C\C^k$.
  The closure of $\set{ \mu(\bxi) : \abs{\mu(\bxi)} = 1}$ is
  \begin{equation*}
    \set{ \mu(\bxi) : \abs{\mu(\bxi)} = 1}
    \cup
    \set*{ \rd_{\bzeta}\mu(\hat\bxi) : \mu(\bzeta) = 0, \abs{\rd_{\bzeta}\mu(\hat\bxi)} = 1 }.
  \end{equation*}
  To see this,
  let $(\epsilon_n^{-1}\bxi_n)$ be a sequence with $\epsilon_n > 0$ and $\abs{\bxi_n} = 1$ and
  \begin{equation*}
    \lim_{n\to\infty} \mu(\epsilon_n^{-1}\bxi_n) = \mu_\bxi.
  \end{equation*}
  After passing to a subsequence,
  $\bxi_n$ converges to a limit $\bxi$ and
  $(\epsilon_n)$ converges to a limit $\epsilon$.
  If $\epsilon \neq 0$,
  then,
  $\mu_\bxi = \mu(\epsilon^{-1}\bxi)$.
  Otherwise,
  $\mu(\bxi) = 0$ and,
  for $n \gg 1$,
  as in the proof of \autoref{Lem_GammaXiMuControlsMu},
  \begin{equation*}
    \bxi_n = \bzeta_n + \hat\bxi_n \qwithq
    \mu(\bzeta_n) = 0 \qandq
    \hat\bxi_n \perp T_{\bzeta_n}\mu^{-1}(0).
  \end{equation*}
  Therefore,
  \begin{equation*}
    \mu(\epsilon_n^{-1}\bxi_n)
    = \rd_{\bzeta_n}\mu(\epsilon_n^{-2}\hat\bxi_n)
    + \epsilon_n^2\mu(\epsilon_n^{-2}\hat\bxi_n).
  \end{equation*}
  By \autoref{Eq_DMuHatXiPerpMuHatXi},
  $\epsilon_n^{-2}\hat\bxi_n$ is bounded;
  hence, after passing to a subsequence,
  it converges to a limit $\hat\bxi$ and
  $\mu_\bxi = \rd_{\bxi}\mu(\hat\bxi)$.
  
  By \autoref{Prop_MuPsiXi=0ImpliesPsi=0} and because $\Psi \neq 0$,
  $\mu_\bxi$ cannot be in $\set{ \mu(\bxi) : \abs{\mu(\bxi)} = 1}$.
  Therefore,
  \begin{equation*}
    \mu_\bxi = \rd_{\bzeta}\mu(\hat\bxi).
  \end{equation*}
  By \autoref{Prop_AbsMuXi=AbsCommutators},
  there exists a maximal torus $\ft \subset \fu(2)$ such that
  \begin{equation*}
    \bzeta \in \ft\otimes \H.
  \end{equation*}
  Therefore,
  \begin{equation*}
    \mu(\Psi) = -\rd_{\bzeta}\mu(\hat\bxi) \in \Im\H\otimes \ft^\perp.
  \end{equation*}
  This, however, is impossible,
  because it would imply that $\Psi = 0$ by \autoref{Prop_AbsPiTMuPsi}.
\end{proof}

\begin{prop}
  \label{Prop_MuPsiXiEstimate}
  There are constants $\delta_\mu,c > 0$ such that,
  for every non-zero pair $\Psi \in \H\otimes_\C\C^2$ and $\bxi \in \H\otimes\su(2)$,
  if
  \begin{equation*}
    \abs{\mu(\Psi,\bxi)} \leq \delta_\mu\paren*{\abs{\Psi}^2 + \abs{\bxi}^2},
  \end{equation*}  
  then
  \begin{equation*}
    \abs{\mu(\Psi,\bxi)}^2
    \leq
      c\paren*{
        \inner{\mu(\Psi,\bxi)}{\mu(\Psi)}
        + \frac{\abs{\Gamma_\bxi\mu(\Psi,\bxi)}^2}{\abs{\Psi}^2+\abs{\bxi}^2}
      }.
  \end{equation*}
\end{prop}

\begin{proof}
  Without loss of generality $\abs{\Psi}^2+\abs{\bxi}^2 = 1$.
  By \autoref{Prop_AbsPiTMuPsi} and \autoref{Prop_MuPsiXiBoundsMuPsi+MuXi},
  \begin{equation*}
    \abs{\Psi}^2 \leq c\abs{\mu(\Psi,\bxi)} \leq c\delta_\mu \qandq
    \abs{\mu(\bxi)} \leq c\abs{\mu(\Psi,\bxi)} \leq c\delta_\mu.
  \end{equation*}
  Therefore,
  for $\delta_\mu \ll 1$,
  \begin{equation*}
    \frac{1}{2} \leq \abs{\bxi} \leq 1.
  \end{equation*}

  If $\delta_\mu \ll 1$,
  then,
  as in the proof of \autoref{Lem_GammaXiMuControlsMu},
  $\bxi$ uniquely decomposes as
  \begin{equation*}
    \bxi = \bzeta + \hat\bxi
  \end{equation*}
  with
  \begin{equation*}
    \mu(\bzeta) = 0, \quad
    \hat\bxi \perp T_{\bzeta}\mu^{-1}(0), \qandq
    \abs{\hat\bxi} \leq c\abs{\mu(\bxi)} \leq c\delta_\mu.
  \end{equation*}
  In particular,
  \begin{equation*}
    \abs{\mu(\hat\bxi)}
    \leq
      c\abs{\mu(\bxi)}^2
    \leq
      c\delta_\mu\abs{\mu(\bxi)}.
  \end{equation*}
  Denote by $\ft \subset \fu(2)$ the maximal torus determined by $\bzeta$ and by $\pi_\ft \co \fu(2) \to \ft$ the orthogonal projection onto $\ft$.
  The discussion in the proof of \autoref{Lem_GammaXiMuControlsMu} shows that
  \begin{equation*}
    \pi_\ft\mu(\bxi) = \mu(\hat\bxi)
    \qandq
    (\one-\pi_\ft)\mu(\bxi) = 2\mu(\bzeta,\hat\bxi),
  \end{equation*}
  and, moreover,
  \begin{equation*}
    \abs{(\one-\pi_\ft)\mu(\Psi,\bxi)}
    \leq
      \abs{\Gamma_{\bzeta}\mu(\Psi,\bxi)}
    \leq
      \abs{\Gamma_\bxi\mu(\Psi,\bxi)} + c\delta_\mu\abs{\mu(\Psi,\bxi)}.
  \end{equation*}

  By the discussion in the preceding paragraph and \autoref{Prop_AbsPiTMuPsi},
  \begin{align*}
    \inner{\mu(\Psi,\bxi)}{\mu(\Psi)}
    &=
      \abs{\pi_\ft\mu(\Psi)}^2
      + \inner{\mu(\hat\bxi)}{\mu(\Psi)}
      + \inner{(\one-\pi_\ft)\mu(\Psi,\bxi)}{\mu(\Psi)} \\
    &\geq
      \abs{\pi_\ft\mu(\Psi)}^2
      - c\delta_\mu\abs{\mu(\Psi,\bxi)}\abs{\Psi}^2
      - c\abs{\Gamma_\bxi\mu(\Psi,\bxi)}\abs{\Psi}^2 \\
    &\geq
      \frac12\abs{\pi_\ft\mu(\Psi)}^2
      - c\delta_\mu^2\abs{\mu(\Psi,\bxi)}^2
      - c\abs{\Gamma_\bxi\mu(\Psi,\bxi)}^2.
  \end{align*}
  Therefore,
  \begin{align*}    
    \abs{\mu(\Psi,\bxi)}^2
    &\leq
      \abs{\pi_\ft\mu(\Psi)}^2
      + \abs{\mu(\hat\bxi)}^2
      + \abs{(\one-\pi_\ft)\mu(\Psi,\bxi)}^2 \\
    &\leq
      c\inner{\mu(\Psi,\bxi)}{\mu(\Psi)}
      + c\delta_\mu^2\abs{\mu(\Psi,\bxi)}^2
      + c\abs{\Gamma_\bxi\mu(\Psi,\bxi)}^2.
  \end{align*}
  For $\delta_\mu \ll 1$, this implies the asserted inequality by rearrangement.
\end{proof}

\subsection{Verification of \autoref{Hyp_GammaFControlsF}}

\begin{lemma}
  \label{Lem_HypothesisForADHM12SW}
  Assume the situation of \autoref{Sec_ADHM12SeibergWittenCompactness}.
  There are constants $r_0,\delta_\mu,c > 0$ such that the following holds for every $x \in M$ and $r \in (0,r_0]$.
  If $A \in \sA(\Ad(\fw))$, $\Psi \in \Gamma(W)$, and $\bxi \in \Gamma(N\otimes\Ad(\fw)_\circ)$ satisfy \autoref{Eq_ADHM12SeibergWitten},
  \begin{equation*}
    \frac12 \leq \sqrt{\abs{\Psi}^2 + \abs{\bxi}^2} \leq 2, \qandq
    \abs{\mu(\Psi,\bxi)} \leq \delta_\mu,
  \end{equation*}
  then
  \begin{equation*}
    \frac{r}{2}\int_{B_{r/2}(x)} \abs{F_{A}}^2
    + \paren*{\frac{r}{\epsilon}}^2\cdot r^{-1}\int_{B_{r/2}(x)} \abs{\nabla_A\Psi}^2
    \leq
      c + cr\int_{B_r(x)} \abs{\Gamma_\bxi F_{A}}^2.
  \end{equation*}
\end{lemma}

\begin{proof}
  By the Lichnerowicz--Weitzenböck formula \autoref{Eq_LichnerowiczWeitzenbock},
  \begin{equation*}
    \frac12\Delta\abs{\Psi}^2
    + \abs{\nabla_A\Psi}^2
    + \epsilon^{-2}\inner{\bgamma(\mu(\Psi,\bxi)\Psi}{\Psi}
    + \inner{\gamma(\fR)\Psi}{\Psi} = 0.
  \end{equation*}  
  Therefore,
  by hypothesis and \autoref{Prop_MuPsiXiEstimate},
  \begin{equation}
    \label{Eq_MuPsiXiWeitzenbock}
    \abs{\mu(\Psi,\bxi)}^2
    + \epsilon^2\abs{\nabla_A\Psi}^2
    \leq
      c_1\paren*{\abs{\Gamma_\bxi\mu(\Psi,\bxi)}^2 + \epsilon^2\abs{\Psi}^2}
    - c_2\epsilon^2\Delta\abs{\Psi}^2.
  \end{equation}
  Let $\chi \in C_0^\infty(B_r(x))$ be a cut-off function satisfying $\chi|_{B_{r/2}(x)} = 1$,
  \begin{equation*}
    r\abs{\nabla\chi} \leq c, \qandq
    r^2\abs{\nabla^2\chi} \leq c;
  \end{equation*}
  in particular,
  \begin{equation*}
    \abs{r^2 \Delta\chi^4} \leq c\chi^2.
  \end{equation*}
  Multiplying \autoref{Eq_MuPsiXiWeitzenbock}
  by $r\epsilon^{-4}\chi^4$,
  integrating by parts,
  and using $F_A = \epsilon^{-2}\mu(\Psi,\bxi)$, yields
  \begin{equation*}
    r\int_{B_r(x)} \chi^4\abs{F_A}^2
    + \paren*{\frac{r}{\epsilon}}^2r^{-1}\int_{B_r(x)} \chi^4 \abs{\nabla_A\Psi}^2 \\
    \leq
      cr\int_{B_r(x)} \chi^4\abs{\Gamma_\bxi F_A}^2
      + c\epsilon^{-2}r^{-1}\int_{B_r(x)} \chi^2\abs{\Psi}^2.
  \end{equation*}
  By \autoref{Prop_AbsPiTMuPsi} and \autoref{Prop_MuPsiXiBoundsMuPsi+MuXi},
  \begin{align*}
    c\epsilon^{-2} r^{-1}\int_{B_r(x)} \chi^2\abs{\Psi}^2
    &\leq
      \frac{r}{c_1 \epsilon^4}\int_{B_r(x)} \chi^4\abs{\Psi}^4
      + c_2  \\
    &\leq
      \frac{r}{2} \int_{B_r(x)} \chi^4\epsilon^{-4}\abs{\mu(\Psi,\bxi)}^2
      + c_2  \\
    &=
      \frac{r}{2} \int_{B_r(x)} \chi^4\abs{F_A}^2
      + c_2.
  \end{align*}
  Plugging this back in to the previous inequality and rearranging proves the asserted inequality.
\end{proof}

\subsection{Conclusion of the proof of \autoref{Thm_ADHM12SeibergWittenCompactness}}
\label{Sec_ConclusionOfProofOfADHM12SeibergWittenCompactness}

Let $(A_n,\Psi_n,\bxi_n)_{n \in \N}$ be a sequence of solutions of \autoref{Eq_ADHM12SeibergWitten} with
\begin{equation*}
  \liminf_{n\to\infty}\, \Abs{(\Psi_n,\bxi_n)}_{L^2} = \infty.
\end{equation*}
Set
\begin{equation*}
  \epsilon_n
  \coloneq
    \frac{1}{\Abs{(\Psi_n,\bxi_n)}_{L^2}}, \quad
  \tilde\Psi_n
  \coloneq
    \epsilon_n\Psi_n, \qandq
  \tilde\bxi_n
  \coloneq
    \epsilon_n\bxi_n.
\end{equation*}

By \autoref{Lem_HypothesisForADHM12SW},
\autoref{Thm_AbstractCompactness} applies to the sequence $(A_n,\tilde\Psi_n,\tilde\bxi_n,\epsilon_n)$.
Therefore and by \autoref{Prop_MuPsiXi=0ImpliesPsi=0},
the following hold:
\begin{enumerate}
\item
  There is a closed, nowhere-dense subset $Z \subset M$,
  a connection $A \in \sA(\Ad(\fw)|_{M\setminus Z},B)$, and
  a section $\bxi \in \Gamma(M\setminus Z,N\otimes\Ad(\fw)_\circ)$ such that the following hold:
  \begin{enumerate}
  \item
    $A$ and $\bxi$ satisfy
    \begin{equation*}
      \begin{split}
        \slD_A\bxi
        &= 0, \\
        \mu(\bxi)
        &=
        0, \qand \\
        \Abs{\bxi}_{L^2}^2
        &=
        1.
      \end{split}
    \end{equation*}
  \item
    The function $\abs{\bxi}$ extends to a Hölder continuous function on all of $M$ and
    \begin{equation*}
      Z = \abs{\bxi}^{-1}(0).
    \end{equation*}
  \end{enumerate}
\item
  After passing to a subsequence and up to gauge transformations,
  for every compact subset $K\subset M\setminus Z$,
  $\paren*{A_n|_K}_{n \in \N}$ converges to $A$ in the weak $W_\loc^{1,2}$ topology,
  $\paren*{\tilde\Psi_n|_K}_{n \in \N}$ converges to $0$ in the weak $W_\loc^{2,2}$ topology,
  $\paren*{\tilde\bxi_n|_K}_{n \in \N}$ converges to $\bxi$ in the weak $W_\loc^{2,2}$ topology,
  and there exists an $\alpha \in (0,1)$ such that $\paren*{\abs{(\tilde\Psi_n,\tilde\bxi_n)}}_{n \in \N}$ converges to $\abs{\bxi}$ in the $C^{0,\alpha}$ topology.
\end{enumerate}

The Euclidean line bundle $\fl$ and the parallel section $\tau$ emerge from the Haydys correspondence \cite[Appendix C]{Doan2017d}.
In the present case this abstract machinery can be made very explicit.
By \autoref{Prop_AbsMuXi=AbsCommutators},
away from $Z$,
$\bxi$ can \emph{locally} be written as
\begin{equation*}
  \bxi = \tau \otimes \nu
\end{equation*}
where $\tau$ is a local section of $\Ad(\fw)_\circ$ which is normalized such that $\abs{\tau} = 1$,
and $\nu$ is a local section of $N$.
This decomposition is unique up to multiplying both $\tau$ and $\nu$ by $-1$.
Decomposing $\bxi$ in this way,
the equation $\slD_A\bxi = 0$ becomes
\begin{equation*}
  0 = \tau\otimes \slD_B\nu + \sum_{i=1}^3 (\nabla_{A,e_i}\tau)\otimes \gamma(e_i)\nu.
\end{equation*}
Since $\tau$ is normalized,
the second term on the right-hand side takes values in $\Span{\tau}^\perp\otimes N$.
Therefore, both summands on the right-hand side vanish separately.
Globally,
there is a flat Euclidean line bundle $\fl$ such that $\nu$ is a section of $N\otimes\fl$ and $\tau$ is a section of $\Hom(\fl,\Ad(\fw)_\circ)$.
The above shows that $(Z,\fl,\nu)$ is a harmonic $\Z_2$ spinor whose zero locus is precisely $Z$ and $\tau$ is $A$--parallel.
By \cite[Theorem 1.4]{Zhang2017},
the former implies the asserted regularity of $Z$.

On every compact subset $K \subset M\setminus Z$,
$\Abs{F_{A_n}}_{L^2(K)}$ is uniformly bounded.
Therefore,
by \autoref{Prop_AbsPiTMuPsi} and \autoref{Prop_MuPsiXiBoundsMuPsi+MuXi},
$\Abs{\Psi_n}_{L^2(K)}$ is uniformly bounded.
Since $\slD_{A_n}\Psi_n = 0$,
it follows that,
possibly after passing to a further subsequence,
$\paren{\Psi_n|_K}$ converges in the weak $W^{2,2}$ topology to a limit $\Psi$ satisfying $\slD_A\Psi = 0$.
For $n \gg 1$, on $K$, we can decompose $\tilde\bxi_n = \bzeta_n + \hat\bxi_n$ as in the proof of \autoref{Lem_GammaXiMuControlsMu}.
Denote by $\ft_n \subset \Ad(\fw)|_K$ the corresponding bundle of maximal tori.
The ADHM$_{1,2}$ Seiberg--Witten equation \autoref{Eq_ADHM12SeibergWitten} and \autoref{Eq_DMuHatXiPerpMuHatXi} imply
\begin{equation*}
  \pi_{\ft_n}F_{A_n} = \pi_{\ft_n}\mu(\Psi_n) + \epsilon_n^{-2}\pi_{\ft_n}\mu(\hat\bxi_n).
\end{equation*}
By \autoref{Eq_DMuHatXiPerpMuHatXi} and \autoref{Prop_MuPsiXiBoundsMuPsi+MuXi},
\begin{equation*}
  \abs{\hat\bxi_n} \leq c\epsilon_n^2\abs{F_{A_n}}.
\end{equation*}
Therefore,
\begin{equation*}
  \abs{\pi_{\ft_n}F_{A_n} - \pi_{\ft_n}\mu(\Psi_n)}
  \leq c\epsilon_n^2\abs{F_{A_n}}^2.
\end{equation*}
From this it follows that
\begin{equation*}
  F_A = \pi_\ft F_A = \pi_\ft\mu(\Psi).
\end{equation*}
This finishes the proof of \autoref{Thm_ADHM12SeibergWittenCompactness}.
\qed
